\documentclass[12pt]{amsart}
\usepackage[usenames,dvipsnames]{color}
\usepackage{fancyvrb}
\usepackage[T1]{fontenc}
\usepackage{lmodern}
\usepackage{upquote,textcomp}
\usepackage{etoolbox}
\usepackage[12pt]{moresize}

\usepackage[latin1]{inputenc}
\usepackage[english]{babel}
\usepackage{amssymb}
\usepackage{amsmath, amsthm, amsfonts, mathrsfs, amsfonts}
\usepackage{latexsym}
\usepackage{textcomp}
\usepackage{enumerate}

\usepackage{graphicx}
\usepackage{tikz}
\usepackage{float}
\usepackage{wrapfig}
\usepackage{array}   

\theoremstyle{definition}
\newtheorem{definition}{Definition}[section]
\newtheorem{definition1}{Definition}

\theoremstyle{remark}
\newtheorem{remark}[definition]{Remark}

\theoremstyle{plain}
\newtheorem{cor}[definition]{Corollary}
\newtheorem{cor1}{Corollary}

\theoremstyle{plain}
\newtheorem{lemma}[definition]{Lemma}

\theoremstyle{plain}
\newtheorem{theorem}[definition]{Theorem}
\newtheorem{theoremA}[]{Theorem}

\theoremstyle{remark}
\newtheorem{notation}[definition]{Notation}

\newcommand{\E}{\mathcal{E}}
\newcommand{\F}{\mathcal{F}}

\newcommand{\G}{\mathrm{G}}
\newcommand{\C}{\mathrm{C}}
\newcommand{\N}{\mathrm{N}}
\newcommand{\Z}{\mathrm{Z}}

\newcommand{\Aut}{\mathrm{Aut}}
\newcommand{\Hom}{\mathrm{Hom}}
\newcommand{\Out}{\mathrm{Out}}
\newcommand{\Inn}{\mathrm{Inn}}
\newcommand{\Syl}{\mathrm{Syl}}

\newcommand{\GF}{\mathrm{GF}}
\newcommand{\GL}{\mathrm{GL}}
\newcommand{\Sp}{\mathrm{Sp}}
\newcommand{\SL}{\mathrm{SL}}
\newcommand{\PGL}{\mathrm{PGL}}
\newcommand{\PSL}{\mathrm{PSL}}

\newcommand{\Alt}{\mathrm{Alt}}

\newcommand{\norm}{\mathrel{\unlhd}}
\def \ov {\overline}

\title{Fusion Systems containing pearls}
\author{Valentina Grazian}
\thanks{Funding: This work was supported by the LMS 150\textsuperscript{th} Anniversary Postdoctoral Mobility Grant [grant number PMG 16-17 10].}
\address{Institute of Mathematics, University of Aberdeen, Fraser Noble Building, Aberdeen AB24 3UE, U.K.}
\email{valentinagrazian@libero.it}

\begin{document}

\begin{abstract} An $\F$-essential subgroup is called a \emph{pearl} if it  is either elementary abelian of order $p^2$ or non-abelian of order $p^3$. In this paper we start the investigation of fusion systems containing pearls: we determine a bound for the order of $p$-groups containing pearls and we classify the saturated fusion systems on $p$-groups containing pearls and having sectional rank at most $4$.
\end{abstract}

\maketitle

\section*{Introduction}
In finite group theory, the word \emph{fusion} refers to the study of conjugacy maps between subgroups of a group. This concept has been investigated for over a century, probably starting with Burnside, and the modern way to solve problems involving fusion is via the theory of fusion systems.
Given any finite group $G$, there is a natural construction of a  saturated fusion system on one of its Sylow $p$-subgroups $S$: this is the category with objects the subgroups of $S$ and with morphisms between subgroups $P$ and $Q$ of $S$ given by the set $\Hom_G(P,Q)$ of all the restrictions of conjugacy maps by elements of $G$ that map $P$ into $Q$. In general, 
a saturated fusion system on a $p$-group $S$ is a category whose objects are the subgroups of $S$ and whose morphisms are the monomorphisms between subgroups which satisfy certain axioms, motivated by conjugacy relations and first formulated in the nineties by the representation theorist Puig (cf. \cite{Pg}). 
There are saturated fusion systems which do not arise as fusion systems of a finite group $G$ on one of its Sylow $p$-subgroups; these fusion systems are called \emph{exotic}. The Solomon fusion systems  $\mathrm{Sol}(p^a)$ (predicted by Benson and studied by Levi and Oliver in \cite{sol}) form the only known family of exotic simple fusion systems on $2$-groups. In contrast, for odd primes $p$, there is a plethora of exotic fusion systems (see for example \cite{ RV, p.index, p.index2, p.index3}). The classification results we prove in this paper lead us to the description of a new exotic fusion system on a $7$-group of order $7^5$.

 The starting point toward the classification of saturated fusion systems is given by the Alperin-Goldschmidt Fusion Theorem  \cite[Theorem 1.19]{AO}, that guarantees that a saturated fusion system  $\F$ on a $p$-group $S$ is completely determined by the $\F$-automorphism group of $S$  and by the $\F$-automorphism group of certain subgroups of $S$, that are called for this reason \emph{$\F$-essential}. 
More precisely, if $\F$ is a saturated fusion system on a $p$-group $S$, then a subgroup $E$ of $S$ is $\F$-essential if
\begin{itemize}
\item $E$ is $\F$-centric: $\C_S(E\alpha) \leq E\alpha$ for every $\alpha \in \Hom_\F(E,S)$;
\item $E$ is fully normalized in $\F$: $|\N_S(E)| \geq |\N_S(E\alpha)|$ for every $\alpha \in \Hom_\F(E,S)$;
\item $\Out_\F(E)=\Aut_\F(E)/\Inn(E)$ contains a strongly $p$-embedded subgroup.
\end{itemize}

The smallest candidate for an $\F$-essential subgroup is a group isomorphic to the direct product $\C_p \times \C_p$, since the outer automorphism group of a cyclic group does not have strongly $p$-embedded subgroups.  The smallest candidate for a non-abelian $\F$-essential subgroup is a non-abelian group of order $p^3$ (that is isomorphic to the group $p^{1+2}_+$ when $p$ is odd). The purpose of this paper is to start the investigation of saturated fusion systems $\F$ containing these small $\F$-essential subgroups.

\begin{definition1}
An \emph{$\F$-pearl} is an $\F$-essential subgroup of $S$ that is either elementary abelian of order $p^2$ or non-abelian of order $p^3$.
\end{definition1}

When this does not lead to confusion, we will omit the $\F$ in front of the name pearl.

 Fusion systems containing pearls are far from being rare. Pearls appear in the study of saturated fusion systems on $p$-groups having a maximal subgroup that is abelian (\cite{p.index, p.index2, p.index3}). Pearls are also contained in many of the saturated fusion systems on a Sylow $p$-subgroup of the group $\G_2(p)$, as proved in \cite{G2p}; in particular the fusion system of the Monster group on one of its Sylow $7$-subgroups contains an abelian pearl. Many saturated fusion systems on $p$-groups of small sectional rank contain pearls as well, when $p$ is odd.  The rank of a finite group $G$ is the minimum size of a generating set for $G$ and a finite $p$-group $S$ has sectional rank $k$ if every elementary abelian quotient $P/Q$ of subgroups of $S$ has order at most $p^k$ (or equivalently if every subgroup of $S$ has rank at most $k$). It turns out that if $p$ is an odd prime then every saturated fusion system $\F$ on a  $p$-group of sectional rank $2$ satisfying $O_p(\F)=1$ contains pearls (\cite[Theorem 1.1]{DRV}). In particular the $\F$-essential subgroups of all the exotic fusion systems on the group $7^{1+2}$ are abelian pearls. In her PhD thesis (\cite{IO}), the author proved that if $p\geq 5$, then every saturated fusion system $\F$ on a $p$-group of sectional rank $3$ satisfying $O_p(\F)=1$ contains an $\F$-essential subgroup that is a pearl. Abelian pearls (and extraspecial pearls modulo their center) are also the smallest example of soft subgroups, defined in \cite{Het1} as self-centralizing abelian subgroups of a $p$-group having index $p$ in their normalizer. Finally pearls are related to the so called $Qd(p)$-groups (as defined in \cite{Glau}). Indeed, suppose $p$ is an odd prime, $E$ is an abelian pearl and $G$ is a model  for $\N_\F(E)$ (as defined in \cite[Section 1.8]{AO}). Then $\N_S(E) \in \Syl_p(G)$, $\Out_\F(E)\cong G/E$ and  the group $O^{p'}(G)$ is a $Qd(p)$-group:
\[ O^{p'}(G) = \langle \N_S(E)^G \rangle \cong (\C_p \times \C_p) \colon\SL_2(p).\]
Similarly, it can be shown (for example by \cite{Winter}) that if $p$ is an odd prime and $E$ is a non-abelian (and so extraspecial) pearl then 
\[ O^{p'}(G)= \langle \N_S(E)^G \rangle \cong p^{1+2}_+\colon \SL_2(p),\]
and we write $\widetilde{Qd}(p)$ to denote groups of this type.
Fusion systems that do not involve $Qd(p)$ groups nor $\widetilde{Qd}(p)$ groups have been studied in \cite{Qdp}.
In such paper the authors also determine the finite simple groups that involve either $Qd(p)$ or $\widetilde{Qd}(p)$ groups (and so the finite simple groups that can realize fusion systems containing pearls). 

We start our analysis by showing that  if $S$ is a $p$-group of order $p^n$ and $\F$ is a saturated fusion system on $S$ containing a pearl, then $S$ has maximal nilpotency class, i.e. it has class $n-1$ (Lemma \ref{self-centr.max.class}). This follows from the well known fact that a $p$-group containing a self-centralizing elementary abelian subgroup of order $p^2$ has maximal nilpotency class (\cite[Proposition 1.8]{Ber}). 
There has been a lot of work on $p$-groups of maximal nilpotency class, aimed to get a better understanding of their structure (see \cite{black}, \cite[III.14]{Hup}, \cite[Chapter 3]{Led-Green}).  One of the main properties of such $p$-groups is that their  upper and lower central series coincide (and they have maximal length). We set $\Z_1(S)=\Z(S)$ and for $i\geq 2$ we denote by $\gamma_i(S)$ and $\Z_i(S)$ the $i$-th member of the lower and upper central series of $S$, respectively (Definition \ref{upper-lower.def}). Also, if $S$ is a $p$-group of maximal nilpotency class and order at least $p^4$, then it contains a characteristic maximal subgroup denoted $\gamma_1(S)$ (see Definition \ref{def.gamma1}) that plays an important role in the classification of fusion systems containing pearls.

If $p=2$ then Lemma \ref{self-centr.max.class} implies that every $2$-group containing a pearl is either dihedral  or quasidihedral (also called semi-dihedral) or generalized quaternion (result that can be deduced directly from \cite[Theorems 1 and 2]{Harada}). So the reduced fusion systems on $2$-groups containing pearls are known (\cite[Theorem A]{2rank4}). 
For this reason we focus on the case in which $p$ is an \emph{odd} prime. Note that if $p$ is odd then every pearl has exponent $p$ (the group $p^{1+2}_-$ cannot be $\F$-essential) and if we denote by $\mathcal{P}(\F)$ the set of $\F$-pearls then we have
\[\mathcal{P}(\F) = \mathcal{P}(\F)_a \cup \mathcal{P}(\F)_e,\]
where $\mathcal{P}(\F)_a$ denotes the set of abelian $\F$-pearls  and $\mathcal{P}(\F)_e$ that of extraspecial $\F$-pearls.

Let $\F$ be a saturated fusion system on the $p$-group $S$ containing a pearl, for $p$ odd. The main result of this paper is the characterization of the order $S$ with respect to its sectional rank.

\begin{theoremA}\label{main}
Suppose that $p$ is an odd prime, $S$ is a $p$-group of sectional rank $k$ and $\F$ is a saturated fusion system on $S$ such that $\mathcal{P}(\F) \neq \emptyset$. Then $S$ has maximal nilpotency class, $p\geq k$ and exactly one of the following holds:
\label{main}
\begin{enumerate}
\item $|S| = p^{k+1}$ and $S$ has a maximal subgroup $M$ that is elementary abelian (and if $|S|\geq p^4$ then $M=\gamma_1(S)$);
\item  $p=k+1$, $|S|\geq p^{p+1}$ and $\gamma_1(S)=\C_S(\Z_2(S))$;
\item  $k \geq 3$, $k+3 \leq p \leq 2k +1$ (with $p=2k +1$ only if $\mathcal{P}(\F)_e = \emptyset$), $S$ has exponent $p$, $\gamma_1(S)$ is not abelian and  $p^{k+2} \leq |S|\leq p^{p-1}$.
\end{enumerate}
\end{theoremA}

The fact that $S$ has sectional rank at most $p$ is consistent with \cite[Theorem A]{elementary}, stating that if $P$ is a $p$-group containing an elementary abelian subgroup of order $p^2$ that is not contained in any other elementary abelian subgroup, then the elementary abelian subgroups of $P$ have order at most $p^p$.

In case $(1)$ of Theorem \ref{main} the $p$-group $S$ is uniquely determined up to isomorphism: it is isomorphic to the semidirect product of an elementary abelian group of order $p^k$ (the subgroup $\gamma_1(S)$ for $|S|\geq p^4$) with a cyclic group of order $p$ acting on it  as a single Jordan block.

If $|S|=p^3$ then $S \cong p^{1+2}_+$ (the group $p^{1+2}_-$ is resistant by \cite[Theorem 4.2]{Stancu}) and $\F$ is among the fusion systems determined in \cite{RV}.
Suppose that $|S|\geq p^4$ and that the group $\gamma_1(S)$ is abelian. Then the reduced fusion systems on $S$ are among the ones studied in \cite{p.index} if $\gamma_1(S)$ is not $\F$-essential, in \cite{p.index2} if $\gamma_1(S)$ is $\F$-essential and elementary abelian and in \cite{p.index3} if $\gamma_1(S)$ is $\F$-essential and not elementary, as might occur only in case (2) of Theorem \ref{main}. Therefore Theorem \ref{main} says that if $p$ is \emph{large enough} then the reduced fusion systems on $p$-groups containing pearls are known.

\begin{cor1}
Suppose that $p$ is an odd prime, $S$ is a $p$-group having sectional rank $k$ and $\F$ is a reduced fusion system on $S$ such that $\mathcal{P}(\F) \neq \emptyset$. If $p > 2k+1$ then $S$ is uniquely determined up to isomorphism and $\F$ is known.
\end{cor1}

In this paper we focus mainly on the cases in which the subgroup $\gamma_1(S)$ is not abelian.

Case $(2)$ of Theorem \ref{main} is the hardest to describe because in such case there is no upper bound for the order of the $p$-group $S$. The investigation of families of $p$-groups of sectional rank $p-1$ containing pearls will be the  subject of future research. Examples of fusion systems of this form are the saturated fusion systems on $3$-groups of maximal nilpotency class and sectional rank $2$ classified in \cite{DRV}.

As for case $(3)$ of Theorem \ref{main}, examples of saturated fusion systems containing (abelian) pearls on a $p$-group $S$ of order $p^{p-1}$ are given in \cite{ParkerStroth} (recall that in this case we have $p\geq 7$). In Lemma \ref{inclusion} we prove that if $\F$ is a saturated fusion system on $S$ containing a pearl $E$, then we can define a saturated fusion subsystem of $\F$ on every subgroup of $S$ properly containing the pearl $E$. As a consequence, for every $p\geq 7$, there are examples of saturated fusion systems containing abelian pearls on a $p$-group $S$ of order $p^a$ for every $3\leq a \leq p-1$ (and if $a\geq k+2$, where $k$ is the sectional rank of $S$, then  we get examples of fusion systems satisfying the assumptions of case $(3)$ of Theorem \ref{main}).

As an application of Theorem \ref{main}, we determine the saturated fusion systems on $p$-groups containing pearls and having sectional rank at most $4$.
In particular we find a new exotic simple fusion system on a $7$-group having sectional rank $3$ and order $7^5$.

\begin{theoremA}\label{small.rank}
 Suppose that $p$ is an odd prime, $S$ is a $p$-group having sectional rank $k\leq 4$ and $\F$ is a saturated fusion system on $S$ such that $\mathcal{P}(\F) \neq \emptyset$. Then $\F$ and $S$ are as described in Table \ref{small.rank.table}.

\begin{table}\renewcommand{\arraystretch}{1.8}
\centering
\normalfont
\resizebox{\textwidth}{!}{
\begin{tabular}{ | c | c | c | c | c |c| }
  \hline			
  $k$ & $p$ & $S$ & Pearls & Possible $\F$-essential subgroups & Saturated fusion systems\\
&&&& that aren't pearls  &\\ \hline \hline
 $2$ & $p\geq 3$ & $S\cong p^{1+2}_+$ & $\C_p \times \C_p$ & none & Classified in \cite{RV} \\ \hline
$2$ & $3$ & $|S|=3^n\geq 3^4$ &  $\C_3 \times \C_3$ or $3^{1+2}_+$ & $\gamma_1(S)$ if $n$ is odd and & Classified in \cite{DRV} \\
&&&& $\gamma_1(S) \cong \C_{(n-1)/2} \times \C_{(n-1)/2}$ &\\ \hline \hline
$3$ & $p \geq 3$ & $|S|=p^4$, $S\in \Syl_p(\Sp_4(p))$ & $\C_p \times \C_p$ and/or $p^{1+2}_+$ & $\gamma_1(S) \cong \C_p \times \C_p \times \C_p$ & Classified in \cite{p.index} and \cite{p.index2}\\
& & $S\in \Syl_p(\Sp_4(p))$ & & & if $\F$ is reduced \\ \hline 
$3$ & $7$ & $|S|=7^5$& $\C_7 \times \C_7$ & none & There is a unique $\F$ (up to iso) \\
& & $S\cong$ \texttt{SmallGroup(7\string^5,37)} & 1 $\F$-class $E^\F$& &  $\F = \F^{(3,7)} = \langle \Aut_\F(S), \Aut_\F(E) \rangle_S$ \\ 
& & $\Out_\F(S)\cong \C_6$ & $\Out_\F(E) \cong \SL_2(7)$ & &simple and exotic (assuming CFSG)\\\hline \hline 
$4$ & $p \geq 5$ & $|S|=p^5$ & $\C_p \times \C_p$ and/or $p^{1+2}_+$ & $\gamma_1(S) \cong \C_p \times \C_p \times \C_p \times \C_p$ &Classified in \cite{p.index} and \cite{p.index2} \\ 
&&&&& if $\F$ is reduced \\ \hline
$4$ & $5$ & $|S|\geq 5^6$ & $\C_5 \times \C_5$ and/or $5^{1+2}_+$ & $\gamma_1(S)=\C_S(\Z_2(S))$ abelian & Classified in \cite{p.index} and \cite{p.index3}\\ 
&&&&of exponent greater than $5$ & if $\F$ is reduced \\ \hline
$4$ & $5$ & $|S|\geq 5^6$ & $\C_5 \times \C_5$ or $5^{1+2}_+$ & $\gamma_1(S)=\C_S(\Z_2(S))$ non-abelian & ?\\ 
& & $\Out_\F(S)\cong \C_4$ & $\Out_\F(E) \cong \SL_2(5)$ & &\\ \hline
$4$ & $7$ & $|S|=7^6$ & $\C_7 \times \C_7$ & $\gamma_1(S) \cong 7^{1+4}_+$ and $\C_S(\Z_2(S))$ & Classified in \cite{G2p}\\ 
& & $S\in \Syl_7(\G_2(7))$ & & &\\ \hline 
$4$ & $7$ & $|S|=7^6$ & $7^{1+2}_+$ & none &  There is a unique $\F$ (up to iso) \\
& & $S\cong$ \texttt{SmallGroup(7\string^6, 813)} & 1 $\F$-class $E^\F$ & & $\F = \F^{(4,7)} = \langle \Aut_\F(S), \Aut_\F(E) \rangle_S$ \\ 
& & $\Out_\F(S)\cong \C_6$ & $\Out_\F(E) \cong \SL_2(7)$ & & $O_7(\F)= \Z(S)$ and $\F/\Z(S) \cong \F^{(3,7)}$ \\
\hline 
\end{tabular}}
\vspace{0.2cm}
\caption{Saturated fusion systems containing pearls on $p$-groups of sectional rank at most $4$, for $p$ odd.}\label{small.rank.table}
\end{table}
\end{theoremA}

The case in which $(k,p)=(4,5)$ and $\gamma_1(S)$ is not abelian is not completed: its study requires more work and will be the subject of future research.
Using the computer software \emph{Magma} we can check that there exist saturated fusion systems containing abelian pearls on the $5$-group stored in \emph{Magma} as \texttt{SmallGroup(5\string^6, i)} for every $i\in\{ 636, 639, 640, 641, 642\}$. These groups are the only candidates for a $5$-group $S$ of sectional rank $4$ and order $5^6$ containing pearls and with no abelian maximal subgroups.

The existence of saturated fusion systems containing pearls on the group stored in \emph{Magma} as   \texttt{SmallGroup(7\string^5, 37)} and on a Sylow $7$-subgroup of $\G_2(7)$ (that have sectional rank $3$ and $4$ respectively),  shows  that the bounds $k+3 \leq p \leq 2k+1$ on $p$ given in part $(3)$ of Theorem \ref{main} are best possible.

Theorem \ref{small.rank} (together with Lemma \ref{inclusion}) implies that if  $\F$ is a saturated fusion system on a Sylow $7$-subgroup $S$ of the group $\G_2(7)$ containing a pearl $E$ and $M$ is a maximal subgroup of $S$ containing $E$, then $\F$ contains a saturated fusion subsystem $\E$ on $M$ that has to be isomorphic to the fusion system $\F^{(3,7)}$ defined in Table \ref{small.rank.table}. Thus we get
\begin{equation}\label{relation} \frac{\F^{(4,7)}}{O_{7}(\F^{(4,7)})} \cong \F^{(3,7)}\cong \E \subset \F.\end{equation}
In particular the simple exotic fusion system $\F^{(3,7)}$ is isomorphic to a subsystem of the $7$-fusion system of the Monster group, since the Monster group and the group $\G_2(7)$ have isomorphic Sylow $7$-subgroups and the $7$-fusion system of the Monster group contains a pearl.

Theorem \ref{main} can also be applied fixing the prime $p$ and letting the sectional rank $k$ vary (recalling that $2\leq k\leq p$). If $p=3$ then $2\leq k\leq 3$ and the fusion systems containing pearls are classified in Theorem \ref{small.rank}:

\begin{cor1}
Let $\F$ be a saturated fusion system on a $3$-group $S$ and suppose that $\mathcal{P}(\F) \neq \emptyset$. Then one of the following holds:
\begin{enumerate}
\item $S\cong 3^{1+2}_+$ and $\F$ is among the fusion systems classified in \cite{RV};
\item $|S|\geq 3^4$, $S$ has sectional rank $2$ and $\F$ is among the fusion systems classified in \cite{DRV}; 
\item $S\cong \C_3 \wr \C_3$ and $\F$, if reduced, is among the fusion systems classified in \cite{p.index} and \cite{p.index2}.
\end{enumerate}
\end{cor1}

Similarly, if $p=5$ then by Theorem \ref{main} we conclude that either $S$ has order at most $5^6$ and contains a maximal subgroup that is elementary abelian or $S$ has sectional rank $4$.
The case $p=7$ is the first admitting examples  of saturated fusion systems containing pearls on $p$-groups as in case (3) of Theorem \ref{main}. In such situation we have $3\leq k\leq 4$ and by Theorem \ref{small.rank} we conclude that either $\F=\F^{(3,7)}$, or  $\F=\F^{(4,7)}$, or $\F$ is a saturated fusion system on a Sylow $7$-subgroup of $\G_2(7)$, and these fusion systems are all related (Equation (\ref{relation})).

\vspace{0.5cm}
\emph{Organization of the paper.} 
In Section 1 we prove some properties of $\F$-essential subgroups, showing in particular that  $p$-groups containing pearls have maximal nilpotency class (Lemma \ref{self-centr.max.class}) and that $\F$-essential subgroups having maximal nilpotency class are pearls (Corollary \ref{max.class.not.ess}).  The more general properties of $\F$-essential subgroups, such as the study of the outer automorphism group of the ones having rank at most $4$, will be used in the characterization of the $\F$-essential subgroups that are not pearls, that is the subject of Section $5$.

In Section $2$ we state some background on $p$-groups having maximal nilpotency class  and we prove some results concerning the relation between the order of such groups and their sectional rank, preparing the field for the proof of Theorem \ref{main}. 

In Section $3$ we study properties of pearls. We start recalling properties of soft subgroups that apply to pearls.  We use these in Theorem \ref{lift} to prove that every pearl $E$ is not properly contained in any $\F$-essential subgroup of $S$ and so every automorphism in $\N_{\Aut_\F(E)}(\Aut_S(E))$ is the restriction of an $\F$-automorphism of $S$. 
We close this section by proving a characterization of the $p$-group $S$ when the subgroup $\gamma_1(S)$ is extraspecial (Theorem \ref{S1extraspecial}) and showing that if the group $\gamma_1(S)$ is not abelian then either all the pearls are abelian or they are all extraspecial (Theorem \ref{type.pearl}).

Section $4$ contains the proof of Theorem \ref{main} and some results concerning the simplicity of $\F$, that will help in the proof of Theorem \ref{small.rank} to show that the fusion system $\F^{(3,7)}$ is simple.

 In Section $5$ we analyze the $\F$-essential subgroups of $p$-groups of maximal nilpotency class that are not pearls, in preparation for the proof of Theorem \ref{small.rank}, that is presented in Section $6$.

\newpage
\section{Preliminaries on Fusion Systems}
Let $p$ be a prime, let $S$ be a $p$-group and let $\F$ be a saturated fusion system on $S$.
We refer to \cite[Chapter 1]{AO} for definitions and notations regarding the theory of fusion systems.

Let $E$ be an $\F$-essential subgroup of $S$. The fact that $E$ is fully normalized in $\F$ guarantees that $\Aut_S(E) \in \Syl_p(\Aut_\F(E))$ and that $E$ is receptive (see \cite[Definition 1.2]{AO}); in particular every automorphism in $\N_{\Aut_\F(E)}(\Aut_S(E))$ is the restriction of an automorphism in $\Aut_\F(\N_S(E))$. The assumption that $\Out_\F(E)$ has a strongly $p$-embedded subgroup implies $O_p(\Out_\F(E))=1$ (or equivalently $O_p(\Aut_\F(E))=\Inn(E)$). We start this section by proving other properties of $\F$-essential subgroups.

\begin{definition}
If $G$ is a group, we say that a morphism $\varphi \in \Aut(G)$ \emph{stabilizes} the series of subgroups  $G_0 \leq G_1 \leq \dots \leq G_n=G$ if $\varphi$ normalizes each $G_i$ and acts trivially on $G_i/G_{i-1}$ for every $1 \leq i \leq n$.
\end{definition}

\begin{notation}
If $P$ is a $p$-group, we write $\Phi(P)$ for the Frattini subgroup of $P$.
\end{notation}

\begin{lemma}\label{char.series}
Let $E\leq S$ be a subgroup of $S$ such that $O_p(\Aut_\F(E))=\Inn(E)$.
Consider the sequence of subgroups:
\begin{equation}\label{series} E_0 \leq E_1 \leq \dots \leq E_n= E \end{equation}
such that $E_0\leq \Phi(E)$ and for every $0 \leq i \leq n$ the group $E_i$ is normalized by $\Aut_\F(E)$. If $\varphi \in \Aut_\F(E)$ stabilizes the series (\ref{series}) then $\varphi \in \Inn(E)$.
\end{lemma}

\begin{proof}
By  \cite[Corollary 5.3.3]{Gor} the order of $\varphi$ is a power of $p$. Note that the set $H$ of all the morphisms in $\Aut_\F(E)$ stabilizing the series (\ref{series}) is a normal $p$-subgroup of $\Aut_\F(E)$. Hence $\varphi  \in H \leq O_p(\Aut_\F(E))$ and by assumption we conclude $\varphi \in \Inn(E)$.
\end{proof}

\begin{lemma}\label{strict.frattini}
Let $E$ be an $\F$-essential subgroup of $S$. Then
\[ \Phi(E) < [\N_S(E),E]\Phi(E) < E.\]
\end{lemma}

\begin{proof}
If $[\N_S(E), E] \leq \Phi(E)$ then the automorphism group $\Aut_S(E)\cong \N_S(E)/\C_S(E)$ centralizes the quotient $E/\Phi(E)$. Hence $\Aut_S(E)=\Inn(E)$ by Lemma \ref{char.series} and $\N_S(E)/\C_S(E) \cong E/\Z(E)$, contradicting the fact that $E$ is $\F$-centric and proper in $S$. So $\Phi(E) < [\N_S(E),E]\Phi(E)$. If $ [\N_S(E),E]\Phi(E) = E$ then $[\N_S(E),E]=E$, contradicting the fact that $S$ is nilpotent. Thus $[\N_S(E),E]\Phi(E) < E$.
\end{proof}

\begin{lemma}\label{self-centr.max.class}
Let $E\leq S$ be a pearl. Then every subgroup of $S$ containing $E$ has maximal nilpotency class. In particular $S$ has maximal nilpotency class.
\end{lemma}

\begin{proof} Let $P$ be a subgroup of $S$ containing $E$. Since pearls have maximal nilpotency class, we may assume that $E < P$.
Note that $E$ is $\F$-centric, so $\C_P(E) \leq \C_S(E) \leq E$.
If $E\cong \C_p \times \C_p$ then $P$ has maximal nilpotency class by \cite[Proposition 1.8]{Ber}. Suppose $E \cong p^{1+2}_+$. Then $\Z(P) = \Z(E)=\Phi(E)$ and $|\Z(P)|=p$. Since $E < P$ and $[\N_S(E) \colon E] =p$, we deduce that $\N_P(E)=\N_S(E)$.
Let $\ov{C}=\C_{P/\Z(P)}(E/\Z(P))$. Then $E/\Z(P) \leq \ov{C} \leq \N_P(E)/\Z(P)=\N_S(E)/\Z(P)$. Suppose by contradiction that $E/\Z(P) < \ov{C}$. Then $\ov{C}=\N_S(E)/\Z(P)=\N_S(E)/\Phi(E)$, contradicting the fact that   $\Phi(E) < [\N_S(E),E]\Phi(E)$ by Lemma \ref{strict.frattini}.  Therefore $E/\Z(P)= \ov{C} =\C_{P/\Z(P)}(E/\Z(P))$. Since $E/\Z(P) \cong \C_p \times \C_p$,  the group $P/\Z(P)$ has maximal nilpotency class by \cite[Proposition 1.8]{Ber}, and the fact that $|\Z(P)|=p$ allows us to conclude that $P$ has maximal nilpotency class.
\end{proof}

As mentioned in the Introduction, the reduced fusion systems on $2$-groups of maximal nilpotency class are known (\cite[Theorem A]{2rank4}).
For this reason, \textbf{from now on we assume that $p$ is an odd prime.}

By definition pearls have maximal nilpotency class.
An application of Lemma \ref{char.series} shows that every $\F$-essential subgroup having maximal nilpotency class is a pearl.

\begin{lemma}\label{max.class.quotient}
Let $E\leq S$ be an $\F$-essential subgroup of $S$. Suppose there exists a subgroup $K$ of $E$ such that
\begin{itemize}
\item $K$ is normalized by $\Aut_\F(E)$;
\item $E/K$ has maximal nilpotency class; and
\item $E < \C_{\N_S(E)}(~ K/(K \cap \Phi(E)) ~)$.
\end{itemize}
Then $E/K$ is isomorphic to either $\C_p \times \C_p$ or $p^{1+2}_+$.
\end{lemma}

\begin{remark}
Note that $[E,K] \leq K \cap [E,E] \leq K \cap \Phi(E)$, so we always have  $E \leq \C_{\N_S(E)}(K/(K \cap \Phi(E))$. The third condition of the previous lemma says that there exists an element $g \in \N_S(E)$ such that the conjugation map $c_g$ is \emph{not} an inner automorphism of $E$ and acts trivially on $K/(K \cap \Phi(E))$. Note in particular that this is true when $K \cap \Phi(E) =K$ (that is $K\leq \Phi(E)$).
\end{remark}

\begin{proof}
Set $\ov{E}=E/K$.
Aiming for a contradiction, suppose $|\ov{E}|=p^m>p^3$.
Let $Z_i$ be the preimage in $E$ of $\Z_i(\ov{E})$ for every $i\geq 1$ and let $C$ be the preimage in $E$ of $\C_{\ov{E}}(\Z_2(\ov{E}))$. Consider the following sequence of subgroups of $E$:
\begin{equation}\label{max.class.seq} K \cap \Phi(E) \leq K < Z_1 < Z_2 < \dots < Z_{m-2} < C < E.\end{equation}
All the subgroups in the sequence are normalized by $\Aut_\F(E)$ (because $K$ is normalized by $\Aut_\F(E)$) and since $\ov{E}$ has maximal nilpotency class every quotient of consecutive members of the sequence, except $K/(K \cap \Phi(E))$, has order $p$.

 Then $\C_{\N_S(E)}(~ K/(K \cap \Phi(E)) ~)$ stabilizes the sequence (\ref{max.class.seq}). Hence $\C_{\N_S(E)}(~ K/(K \cap \Phi(E)) ~) \leq \Inn(E)$ by Lemma \ref{char.series}, a contradiction.

Hence we have $|\ov{E}|\leq p^3$.
If $|\ov{E}| =p$ then $\Phi(E) \leq K$ and $\C_{\N_S(E)}(~ K/(K \cap \Phi(E)) ~)$ stabilizes the sequence $\Phi(E) < K < E$, giving again a contradiction by Lemma \ref{char.series}. Thus $p^2 \leq |\ov{E}| \leq p^3$.

Since $\ov{E}$ has maximal nilpotency class, then either $\ov{E}$ is abelian of order $p^2$ or $\ov{E}$ is extraspecial of order $p^3$. Moreover $\ov{E}$ has exponent $p$, otherwise we can consider the sequence $K \cap \Phi(E) \leq K \leq K\Phi(E) < K\Omega_1(E) < E$ and we get a contradiction by Lemma \ref{char.series}.
Thus either  $\ov{E} \cong \C_p \times \C_p$ or $\ov{E} \cong p^{1+2}_+$.
\end{proof}

A direct consequence of Lemma  \ref{max.class.quotient} applied with $K=1$ is the following

\begin{cor}\label{max.class.not.ess}
Let $E\leq S$ be an $\F$-essential subgroup of $S$. If $E$ has maximal nilpotency class
then $E$ is a pearl.
\end{cor}

The next result is about the outer automorphism group of an $\F$-essential subgroup.

\begin{lemma}\label{GLr}
Let $E$ be an $\F$-essential subgroup of $S$. Then $\Out_\F(E)$ acts faithfully on $E/\Phi(E)$. In particular if $E/\Phi(E)$ has order $p^r$  then $\Out_\F(E)$ is isomorphic to a subgroup of $\GL_r(p)$.
\end{lemma}

\begin{proof}
By Lemma \ref{char.series} we get $\C_{\Aut_\F(E)}(E/\Phi(E)) = \Inn(E)$. Hence  the group $\Out_\F(E)\cong \Aut_\F(E)/\Inn(E)$ acts faithfully on $E/\Phi(E)$. Since $E/\Phi(E)$ is elementary abelian (\cite[Theorem 5.1.3]{Gor}) we have $\Aut(E/\Phi(E)) \cong \GL_r(p)$, and we conclude.
\end{proof}

The following theorem characterizes the automorphism group of $\F$-essential subgroups that have rank at most $3$.
When we write $A \leq \Out_\F(E) \leq B$ we mean that $\Out_\F(E)$ is isomorphic to a subgroup of $B$ and contains a subgroup isomorphic to $A$.

\begin{theorem}\label{auto.rank.3}
Let $E\leq S$ be an $\F$-essential subgroup. If $E$ has rank at most $3$ then one of the following holds:
\begin{enumerate}
\item  $|E/\Phi(E)|= p^2$ and $\SL_2(p) \leq \Out_\F(E) \leq \GL_2(p)$;
\item $|E/\Phi(E)|= p^3$,  the action of $\Out_\F(E)$ on $E/\Phi(E)$ is reducible and
\[ \SL_2(p) \leq \Out_\F(E) \leq \GL_2(p) \times \GL_1(p);\]
\item  $|E/\Phi(E)|= p^3$,  the action of $\Out_\F(E)$ on $E/\Phi(E)$ is irreducible and the group  $O^{p'}(\Out_\F(E))$ is isomorphic to one of the following groups:
\begin{enumerate}
\item $\SL_2(p)$;
\item $\PSL_2(p)$;
\item the Frobenius group $13: 3$ with $p=3$.
\end{enumerate}
\end{enumerate}
In particular $[\N_S(E) \colon E]=p$ and every $P\in E^\F$ is $\F$-essential.
\end{theorem}

\begin{proof}
Set $G=\Out_\F(E)$.
Let $r$ be the rank of $E$. Then $2\leq r \leq 3$ and by Lemma \ref{GLr}  the group $G$ is isomorphic to a subgroup of $\GL_r(p)$.
\begin{enumerate}
\item Suppose $|E/\Phi(E)|=p^2$. Thus the Sylow $p$-subgroups of $G$ have order at most $p$. Since $G$ has a strongly $p$-embedded subgroup, we deduce that it contains at least two Sylow $p$-subgroups. Hence $O^{p'}(G) \cong \SL_2(p)$ (\cite[Theorem 2.8.4]{Gor}).  In particular the quotient $\N_S(E)/E \cong \Out_S(E)$ has order $p$. Also, by Lemma \ref{strict.frattini} we deduce that $E/\Phi(E)$ is a  natural $\SL_2(p)$-module for $O^{p'}(G)=\langle \Out_S(E)^{G} \rangle$. 

\item Suppose $|E/\Phi(E)|= p^3$ and the action of $G$ on $E/\Phi(E)$ is reducible.
Set $V=E/\Phi(E)$. Let $U$ be a proper subgroup of $V$ normalized by $G$. Then $U$ is normalized by $O^{p'}(G)$ and $p \leq |U| \leq p^2$.

\begin{enumerate}
\item Suppose $|U|=p$. Then $[S,U]=1$ for every Sylow $p$-subgroup $S$ of $G$. Thus $[O^{p'}(G),U]=1$. So the subgroup $\C_{O^{p'}(G)}(V/U)$ stabilizes the sequence $1 < U < V$. Since $G$ has a strongly $p$-embedded subgroup we have $O_p(O^{p'}(G))=1$ and so $\C_{O^{p'}}(V/U)=1$.
Therefore $O^{p'}(G) \hookrightarrow \Aut(V/U)\cong \GL_2(p)$ and so $O^{p'}(G) \cong \SL_2(p)$.

Let $t \in \Z(O^{p'}(G))$ be an involution. Then by coprime action we get
\[ V = [V, t] \times \C_V(t).\]
Note that the subgroups $[V, t]$ and $\C_V(t)$ of $V$ are normalized by $G$.
Also, $U \leq\C_V(t)$. If $U \neq \C_V(t)$, then the quotients $V/\C_V(t)$ and $\C_V(t)/U$ have dimension 1. In particular $O^{p'}(G)$ stabilizes the series $1 < U < \C_V(t) < V$ and we get a contradiction by Lemma  \ref{char.series}. Hence $U=\C_V(t)$, $[V,t]$ has order $p^2$ and $V = [V,t] \times U$.

\item Suppose $|U|=p^2$. Note that $V$ can be seen as a $3$-dimensional vector space over $\GF(p)$ and the group $G$ acts on the dual space $V^*$, that is a $3$-dimensional vector space over $\GF(p)$ as well. Also, since $G$ normalizes $U$, it normalizes the subspace
\[ U^{\perp}= \{ \varphi \in V^* \mid u\varphi = 0 \text{ for every } u\in U\} \subseteq V^*.\]
Note that $U^{\perp}$ has dimension $1=\dim V - \dim U$. Thus $G$ normalizes a $1$-dimensional subspace of a vector space of dimension $3$. Hence, with an argument similar to the one used in part (a), we can show that there exists a $2$-dimensional subspace $W^*$ of $V^*$ normalized by $G$ and such that $V^* = U^* \oplus W^*$. In particular the corresponding subspace $W=(W^*)^{\perp}$ of $V$  is a $1$-dimensional subspace normalized by $G$ and such that $V = U \oplus W$.
\end{enumerate}

Since $O_p(G)=1$, there are unique subgroups $U\cong \C_p \times \C_p$ and $W \cong \C_p$ of $V$ that are normalized by $G$.
Therefore $G \leq \GL(U) \times \GL(W) \cong \GL_2(p) \times \GL_1(p)$.

\item Suppose $|E/\Phi(E)|= p^3$ and the action of $G$ on $E/\Phi(E)$ is irreducible. If $3\leq p\leq 5$ then we can prove the statement using the computer software \emph{Magma}. Suppose $p\geq 7$. If the action of $O^{p'}(G)$ on $E/\Phi(E)$ is reducible, then we can repeat the argument used in part (2) to conclude that $O^{p'}(G)\cong \SL_2(p)$.

Suppose the action of $O^{p'}(G)$ on $E/\Phi(E)$ is irreducible. Note that $O^{p'}(G) \leq \SL_3(p)$ and $K=O^{p'}(G)\Z(\SL_3(p))/\Z(\SL_3(p))$ is a subgroup of $\PSL_3(p)$ having a strongly $p$-embedded subgroup.
Using  the classification of maximal subgroups of $\PSL_3(p)$ appearing in \cite[Theorem 6.5.3]{GLS3} and the fact that $p$ divides the order of $K$ and $O_p(K)=1$, we deduce that either  $K$ is isomorphic to a subgroup of $\PGL_2(p)$ or $p=7$ and $K$ is isomorphic to a subgroup of $\PSL_3(2)\cong \PSL_2(7)$. By the definition of $K$ we conclude that $K \cong \PSL_2(p)$ for every $p\geq 7$. Finally note that the group $\PSL_2(p)$ has a Schur multiplier of order $2$ (\cite[Satz V.25.7]{Hup}) and  $|\Z(\SL_3(p))|$ is either $1$ or $3$. Hence $O^{p'}(G) \cong K$.
\end{enumerate}

Note that $\N_S(E)/E \cong \Out_S(E) \in \Syl_p(\Out_\F(E))$. Thus, what is  proved in all three cases implies $[\N_S(E) \colon E] =p$.

Suppose $P \in E^\F$. Then $P$ is $\F$-centric and $\Out_\F(P) \cong \Out_\F(E)$ has a strongly $p$-embedded subgroup. Since $E$ is fully normalized and $P$ is a proper subgroup of $S$ we have $|\N_S(E)| \geq |\N_S(P)| > |P|=|E|$. Since $[\N_S(E) \colon E] =p$ we deduce $|\N_S(P)|=|\N_S(E)|$ and so $P$ is fully normalized. Therefore $P$ is $\F$-essential.
\end{proof}

 Theorem \ref{auto.rank.3}(1) applies to pearls.

\begin{cor}\label{auto.pearl}
Let $E\in \mathcal{P}(\F)$ be a pearl. Then $\Out_\F(E) \leq \GL_2(p)$, $O^{p'}(\Out_\F(E))\cong \SL_2(p)$, $[\N_S(E)\colon E] =p$ and every subgroup of $S$ that is $\F$-conjugate to $E$ is a pearl.
\end{cor}

As for essential subgroups $E$ having rank $4$, the fact that $\Out_\F(E)$ is isomorphic to a subgroup of $\GL_4(p)$ (Lemma \ref{GLr}) combined with \cite[Theorems 6.4 and 6.9]{Sambale}, prove the following.

\begin{theorem}\label{samb4}
Let $E$ be an $\F$-essential subgroup of $S$ having rank $4$. Then one of the following holds:
\begin{enumerate}
\item $[\N_S(E) \colon E]=p$;
\item $\N_S(E)/E \cong \C_9$ and  $p=3$;
\item $[\N_S(E) \colon E]=p^2$ and $O^{p'}(\Out_\F(E)/O_{p'}(\Out_\F(E))) \cong \PSL_2(p^2)$.
\end{enumerate}
\end{theorem}

More information on the outer automorphism group of $\F$-essential subgroups having index $p$ in their normalizer is given by the next lemma.

\begin{lemma}\label{SL.index.p}
Let $E$ be an $\F$-essential subgroup of $S$. Suppose that $\N_S(E) < S$, $[\N_S(E) \colon E] =p$ and $\Phi(E) \norm \N_S(\N_S(E))$. Then $O^{p'}(\Out_\F(E)) \cong \SL_2(p)$ and if we set $\ov{E}= E/\Phi(E)$, then $\ov{E}/\C_{\ov{E}}(O^{p'}(\Out_\F(E)))$ is a natural $\SL_2(p)$-module for  $O^{p'}(\Out_\F(E))$.
\end{lemma}

\begin{proof}
Set $N=\N_S(E)$ and take $x\in \N_S(N) \backslash N$. Then $N=EE^x$ and by assumption $\Phi(E)=\Phi(E)^x=\Phi(E^x)$. Thus $E^x/\Phi(E)$ is abelian and so $(E \cap E^x)/\Phi(E) = \Z(N/\Phi(E))$. Set $\ov{N}=N/\Phi(E)$, $\ov{E}=E/\Phi(E)$, $G=O^{p'}(\Out_\F(E))$ and $A=\Out_S(E)\cong N/E$. Note that $G$ acts faithfully on $\ov{E}$ and both  $A/\C_A(\ov{E}) \cong A$ and $\ov{E}/ \C_{\ov{E}}(A) = \ov{E}/\Z(\ov{N})$ have order $p$. Hence the group $A$ is an offender for $\ov{E}$ in $G$, that is a subgroup $P$ of $G$ such that $P/C_P(\ov{E})$ is non-trivial and $|\ov{E}/\C_{\ov{E}}(P)| \leq |P/\C_P(\ov{E})|$.  Since $E$ is $\F$-essential, the group $G$ has a strongly $p$-embedded subgroup. Hence by \cite[Theorem 5.6]{Ellen} we conclude that $G\cong \SL_2(p)$ and $\ov{E}/\C_{\ov{E}}(G)$ is a natural $\SL_2(p)$-module for $G$.
\end{proof}

\section{On $p$-groups having maximal nilpotency class}
We refer to \cite[III.14]{Hup} and  \cite[Chapter 3]{Led-Green}  for definition and properties of $p$-groups of maximal nilpotency class. In this section we introduce our notation, we state the facts that we are going to use the most and we prove some new results.

Let $p$ be an odd prime and let $S$ be a $p$-group having maximal nilpotency class and order $|S|=p^n \geq p^3$.

\begin{definition}\label{upper-lower.def}~
\begin{itemize}
\item Set $\gamma_2(S) = [S,S]$ and $\gamma_i(S)=[\gamma_{i-1}(S),S]$ for every $i\geq3$.
\item Set $\Z_1(S)=\Z(S)$ and for every $i\geq 2$ let $\Z_i(S)$ be the preimage in $S$ of $\Z(S/\Z_{i-1}(S))$.
\end{itemize}
The groups $\gamma_i(S)$ are the members of the lower central series of $S$ and the groups $\Z_i(S)$ form the upper central series of $S$.
\end{definition}

Since $S$ has maximal nilpotency class we deduce that $\gamma_{n}(S)=1$ and $\Z_{n-1}(S)=S$. Also $S/\gamma_2(S) \cong \C_p \times \C_p$ and for every $2\leq i \leq n-1$ we have $\gamma_i(S)=\Z_{n-i}(S)$ and  $[\gamma_i(S) \colon \gamma_{i+1}(S)]=p$.
If $N$ is a normal subgroup of $S$ having order $p^i\leq p^{n-2}$ then $N=\gamma_{n-i}(S)=\Z_i(S)$ (\cite[Hilfssatz III.14.2]{Hup}).

\begin{definition}\label{def.gamma1}
If $|S|\geq p^4$ we write $\gamma_1(S)$ for the centralizer in $S$ of the quotient $\gamma_2(S)/\gamma_{4}(S)$.
\end{definition}

Note that $\gamma_1(S)$ is a maximal subgroup of $S$ that is characteristic in $S$. Also, $\gamma_2(S) < \gamma_1(S) < S$. In this sense the group $\gamma_1(S)$ \emph{completes} the lower central series of $S$.

\begin{lemma}\cite[Hauptsatz III.14.6]{Hup} Suppose $|S|=p^n\geq p^5$. Then
 \[ \gamma_1(S)=\C_S(\gamma_i(S)/\gamma_{i+2}(S)) \text{ for every } 2\leq i \leq n-3. \]
\end{lemma}

Note that the group $\gamma_1(S)$ can equal $\C_S(\gamma_{n-2}(S)/\gamma_{n}(S)) =\C_S(\Z_2(S))$. The next theorem tells us when this can happen.

\begin{theorem}\label{equal2step} \cite[Corollary 3.2.7, Theorem 3.2.11, Theorem 3.3.5]{Led-Green} Suppose $|S|=p^n \geq p^4$
 and assume one of the following holds:
\begin{enumerate}
\item $n=4$; or
\item $n > p+1$; or
\item $5 \leq n \leq p+1$ and $n$ is odd.
\end{enumerate}
Then $\gamma_1(S) =\C_S(Z_2(S))$.
 \end{theorem}

\begin{definition}
Suppose $|S|=p^n\geq p^4$.
The \emph{degree of commutativity} of $S$ is the largest integer $l$ such that $[\gamma_i(S),\gamma_j(S)] \leq \gamma_{i+j+l}(S)$ for every $1\leq i,j \leq n$ if $\gamma_1(S)$ is not abelian,  and it is equal to $n-3$ if $\gamma_1(S)$ is abelian.
\end{definition}

\begin{theorem}\label{positive.deg}\cite[Theorem 3.2.6]{Led-Green}
Let $l$ be the degree of commutativity of $S$. Then $l\geq 0$ and
\[ l \geq 1 \Leftrightarrow \gamma_1(S)=\C_S(\Z_2(S)).\]
\end{theorem}

The following theorem gives an upper bound for the order of $S$ with respect to the degree of commutativity of $S$.

\begin{theorem}\label{bound.deg.comm}\cite[Theorem 3.4.11, Corollary 3.4.12]{Led-Green}
Suppose that $p\geq 5$ and $|S|=p^n$ and let $l$ be the degree of commutativity of $S$. Then
\begin{enumerate}
\item $n \leq 2l + 2p -4$;
\item if $i= \lceil (2p-5)/3 \rceil$ then the group $\gamma_i(S)$ has nilpotency class at most $2$;
\item the group $[\gamma_1(S), \gamma_1(S), \gamma_1(S)]$ has order at most $p^{2p-8}$.
\end{enumerate}
\end{theorem}

The power structure of the members of the lower central series of $S$ is also known.

\begin{theorem}\label{power}\cite[Proposition 3.3.2, Corollary 3.3.6]{Led-Green}
Suppose $|S|=p^n \geq p^4$. Then
\begin{itemize}
\item if $n\leq p+1$ then $\gamma_2(S)$  has exponent $p$ and $S^p \leq \Z(S)$;
\item if $n> p+1$ then for every $1 \leq  i \leq n-(p-1)$ we have $\Omega_1(\gamma_i(S))=\Z_{p-1}(S)$ and $\gamma_i(S)^p = \gamma_{i+p-1}(S)$.
\end{itemize}
\end{theorem}

Note that the previous theorem also implies that the group $\Z_i(S)$ has exponent $p$ for every $i\leq $ \rm{min}$\{ p - 1, n - 2\}$. Also, the group $\Omega_1(\gamma_1(S))$ has exponent $p$ (since it is generated by elements of order $p$ and has order at most $p^{p-1}$).

\begin{cor}\label{omegas}
Suppose $|S|=p^n > p^{p+1}$. Let $m\in \mathbb{N}$ be such that $\Omega_m(\gamma_1(S))< \Omega_{m+1}(\gamma_1(S))=\gamma_1(S)$. Then
\[ [\Omega_j(\gamma_1(S)) \colon \Omega_{j-1}(\gamma_1(S))] = p^{p-1} \text{ for every } 2 \leq j\leq m \text{ and } [\gamma_1(S) \colon \Omega_m(\gamma_1(S))] \leq p^{p-1}.\] Also, for every $j\leq m+1$ the quotient $\Omega_{j}(\gamma_1(S))/\Omega_{j-1}(\gamma_1(S))$ has exponent $p$.
\end{cor}

Since the members of the lower central series of $S$ are the only normal subgroups of $S$ contained in $\gamma_1(S)$, we get $\Omega_j(\gamma_1(S))=\gamma_{n-j(p-1)}(S)=\Z_{j(p-1)}(S)$ for every $j\leq m$.

\begin{proof}
We prove the statement by induction on $j\leq m$.
If $j=1$ then $\Omega_1(\gamma_1(S))=\Z_{p-1}(\gamma_1(S))$ has order $p^{p-1}$ by Theorem \ref{power}.
Suppose $[\Omega_j(\gamma_1(S)) \colon \Omega_{j-1}(\gamma_1(S))] = p^{p-1}$ and $\Omega_j(\gamma_1(S))<\gamma_1(S)$.

Suppose $[S \colon \Omega_j(\gamma_1(S))] \geq p^{p+2}$. Then we can apply Theorem \ref{power} to the quotient $S/\Omega_j(\gamma_1(S))$, that has maximal nilpotency class, and we conclude that $[\Omega_{j+1}(\gamma_1(S)) \colon \Omega_{j}(\gamma_1(S))]=p^{p-1}$. Since the quotient $\Omega_{j+1}(\gamma_1(S))/\Omega_{j}(\gamma_1(S))$ is generated by elements of order $p$ and has order $p^{p-1}$, we also deduce that it has exponent $p$.

Suppose $[S \colon \Omega_j(\gamma_1(S))] \leq p^{p+1}$.
Thus $\Omega_{j+1}(\gamma_1(S))/\Omega_j(\gamma_1(S))$ has order at most $p^p$ and it is generated by elements of order $p$. So the quotient group $\Omega_{j+1}(\gamma_1(S))/\Omega_j(\gamma_1(S))$ has exponent $p$. By Theorem \ref{power} we have $\gamma_2(S)\leq \Omega_{j+1}(\gamma_1(S))$ and $\gamma_1(S)^p =\gamma_{p}(S)$.
If $[\gamma_1(S) \colon \Omega_j(\gamma_1(S))]\leq p^{p-1}$ then $\gamma_{p}(S) \leq \Omega_j(\gamma_1(S))$ and so $j=m$.
Assume $[\gamma_1(S) \colon \Omega_j(\gamma_1(S))]=p^p$. If $\Omega_{j+1}(\gamma_1(S)) = \gamma_1(S)$ then $\gamma_1(S)/\Omega_j(\gamma_1(S))$ has exponent $p$, contradicting the fact that $\gamma_1(S)^p =\gamma_{p}(S) \nleq \Omega_j(\gamma_1(S))=\gamma_{p+1}(S)$. Therefore $\gamma_2(S)=\Omega_{j+1}(\gamma_1(S))$, $[\Omega_{j+1}(\gamma_1(S)) \colon \Omega_j(\gamma_1(S))] =p^{p-1}$ and $[\gamma_1(S) \colon \Omega_{j+1}(\gamma_1(S))] = [\gamma_1(S)\colon \gamma_2(S)] =p \leq p^{p-1}$.
\end{proof}

\begin{lemma}\label{omega.series}
Suppose that $|S|=p^n > p^{p+1}$ and let $l$ be the degree of commutativity of $S$. Fix $i\geq 1$ and consider the action of $\gamma_i(S)$ on $\gamma_1(S)$ by conjugation. If $l\geq (p-1) - i$ then $\gamma_i(S)$ stabilizes the series
\[ 1=\Omega_0(\gamma_1(S)) < \Omega_1(\gamma_1(S)) < \Omega_2(\gamma_1(S)) < \dots \Omega_m(\gamma_1(S)) < \Omega_{m+1}(\gamma_1(S))=\gamma_1(S).\]
\end{lemma}

\begin{proof}
By Corollary \ref{omegas} we have $[\Omega_j(\gamma_1(S)) \colon \Omega_{j-1}(\gamma_1(S))]\leq p^{p-1}$ for every $1\leq j \leq m+1$. In other words, if $\Omega_j(\gamma_1(S))=\gamma_{k}(S)$ for some $k\geq 1$, then $\gamma_{k+(p-1)}(S) \leq \Omega_{j-1}(\gamma_1(S))$.
Hence $[\gamma_i(S), \Omega_j(\gamma_1(S))] = [\gamma_i(S), \gamma_{k}(S)] \leq \gamma_{i+k+ l}(S) \leq \gamma_{k + (p-1)}(S) \leq \Omega_{j-1}(\gamma_1(S))$. Thus $\gamma_i(S)$ centralizes the quotient $\Omega_j(\gamma_1(S))/\Omega_{j-1}(\gamma_1(S))$ for every $1 \leq j \leq m+1$.
\end{proof}

We now consider elements and  subgroups of $S$ not contained in $\gamma_1(S)$.

\begin{lemma}\label{exponent}
Suppose $|S|=p^n$ and $x\in S$ is not contained in $\gamma_1(S)$. Then
\begin{enumerate}
\item $x^p \in \Z_2(S)$ and if $x\notin \C_S(\Z_2(S))$ then $x^p \in \Z(S)$;
\item for every $i\geq 1$, if $s_i\in \gamma_i(S) \backslash \gamma_{i+1}(S)$ and either $x\notin \C_S(\Z_2(S))$ or $\Z(S)\leq \langle x, s_i \rangle$, then $\gamma_i(S) \leq \langle x, s_i \rangle $.
\end{enumerate}
\end{lemma}

\begin{proof} ~
\begin{enumerate}
\item Suppose $x^p\neq 1$ and let  $1\leq i \leq n-1$ be such that $x^p \in \gamma_i(S) \backslash \gamma_{i+1}(S)$. So $\gamma_i(S)= \langle x^p\rangle \gamma_{i+1}(S)$. We want to prove that $i\geq n-2$. Note that
\[ [\gamma_i(S), x] = [\gamma_{i+1}(S)\langle x^p \rangle, x] = [\gamma_{i+1}(S),  x] \leq [\gamma_{i+1}(S),S] = \gamma_{i+2}(S).\]
Hence $x \in \C_S(\gamma_i(S)/\gamma_{i+2}(S))$. Since $x\notin \gamma_1(S)$ we deduce that $i \geq n-2$. So $x^p\in \gamma_{n-2}(S)=\Z_2(S)$.
If $x\notin \C_S(\Z_2(S))=\C_S(\gamma_{n-2}(S)/\gamma_{n}(S))$ then we also have $i \geq n-1$ and $x^p\in \gamma_{n-1}(S)=\Z(S)$.

\item Let $i\geq 1$ and set $s_{j+1}=[x, s_j] \in \gamma_{j+1}(S)$ for $i \leq j \leq n-2$. Since $x$ is not contained in  $\gamma_1(S)$, the element $x$ does not centralize $\gamma_j(S)/\gamma_{j+2}(S)$ for every $j\leq n-3$. So $s_j \in \gamma_j(S) \backslash \gamma_{j+1}(S)$, and so $\gamma_j(S)=\gamma_{j+1}(S)\langle s_j \rangle$ for every $j\leq n-2$.  If $x\notin \C_S(\Z_2(S))$ then $1\neq s_{n-1}=[x,s_{n-2}] \in \Z(S)$; so $\Z(S)=\langle s_{n-1} \rangle \leq \langle x, s_i\rangle$. Thus in any case we have  $\Z(S)\leq \langle x, s_i \rangle$ and $\gamma_i(S) = \langle s_i, s_{i+1}, \dots , s_{n-2} \rangle \Z(S) \leq \langle x, s_i \rangle$.
\end{enumerate}
\end{proof}

\begin{lemma}\label{sbg.max.class} Suppose $|S|=p^n$ and let $P$ be a proper subgroup of $S$ of order $p^m$ such that $P$ is not contained in $\gamma_1(S)$. Suppose moreover that either $P\nleq \C_S(\Z_2(S))$ or $\Z(S)\leq P$.
Then either $|P|=p$ or  $\Z_{m-1}(S) \leq P$ and $[P \colon \Z_{m-1}(S)]=p$. Also
\begin{itemize}
\item if $P \nleq \C_S(\Z_2(S))$ and $|P|\geq p^2$ then $P$ has maximal nilpotency class;
\item if $P\leq \C_S(\Z_2(S))$ and $|P| \geq p^3$ then $P/\Z(S)$ has maximal nilpotency class.
\end{itemize}
\end{lemma}

\begin{proof}
Note that $P\gamma_1(S) = S$ and $[S\colon \gamma_1(S)]=p$, so $[P \colon P \cap \gamma_1(S)]=p$.

If $P\cap \gamma_1(S) = 1$ then $|P| =p$.
Suppose there exists $1 \neq z \in P \cap \gamma_1(S)$. Then there exists $i$ such that $z \in \gamma_i(S) \backslash \gamma_{i+1}(S)$.
In other words $\gamma_i(S)=\langle z \rangle \gamma_{i+1}(S)$.
\begin{enumerate}
\item Suppose $P\nleq \C_S(\Z_2(S))$ and let $x\in P$ be such that $x$ is not contained in $\gamma_1(S)$ nor $\C_S(\Z_2(S))$. Then by Lemma \ref{exponent} we have $\gamma_i(S) \leq \langle z , x \rangle \leq P$.
Let $j\in \mathbb{N}$ be minimal such that $\gamma_j(S) \leq P$. If $y\in P \backslash \gamma_j(S)$ then $y \notin \gamma_1(S)$ (otherwise $y\in \gamma_{k}(S)$ for some $k < j$ and $\gamma_{k}(S)\leq P$ contradicting the minimality of $j$). So $[P \colon \gamma_j(S)] = [S \colon \gamma_1(S)] =p$ and $\gamma_j(S) = P \cap \gamma_1(S)$. In particular $|\gamma_j(S)|=p^{m-1}$ and so $j=n-(m-1)$ and $\gamma_j(S)=\Z_{m-1}(S)$. Using the fact that $x$ is in neither $\gamma_1(S)$ nor $\C_S(\Z_2(S))$, we conclude that $\gamma_k(P)=\gamma_{n-m+k}(S)$ for every $k\geq 1$ and so $P$ has maximal nilpotency class.

\item Suppose $P\leq \C_S(\Z_2(S))$.
The group $\ov{S}=S/\Z(S)$ is a $p$-group of maximal nilpotency class. Also, $\gamma_i(\ov{S}) = \gamma_i(S)/\Z(S)$ for every $i$ and $\Z_2(\ov{S})=\Z_3(S)/\Z(S)$.
Thus $\gamma_1(\ov{S})=\C_{\ov{S}}(\Z_2(\ov{S}))$.
By assumption $\Z(S) \leq P$ so we can consider the group $\ov{P}=P/\Z(S) \leq \ov{S}$. Note that $\ov{P}$ is not contained in $\gamma_1(\ov{S})$. So by part (1) we conclude that $[\ov{P} \colon \gamma_{n-(m-1)}(\ov{S})] =p$ and either $|\ov{P}| =p$ or $\ov{P}$ has maximal nilpotency class.
\end{enumerate}
\end{proof}

\begin{remark}
If $G$ is a group and $N$ is a normal subgroup of $G$ such that $G/N$ is cyclic then $[G,G]=[G,N]$ (\cite[Lemma 2.1]{black}). This fact will be used several times in this paper, especially applied with $G=\gamma_i(S)$ and $N=\gamma_{i+1}(S)$ for $i\geq 1$.
\end{remark} 

A direct consequence of Lemma  \ref{sbg.max.class} is the following.

\begin{cor}\label{max.sbg.no.max.class}
If $|S|\geq p^4$ then $\gamma_1(S)$ and $\C_S(\Z_2(S))$ are the only maximal subgroups of $S$ that do not have maximal nilpotency class.
\end{cor}

\begin{proof}
Note that $\Z_2(S) \leq \Z(\C_S(\Z_2(S)))$ and $|\Z_2(S)| =p^2$, so $\C_S(\Z_2(S))$ does not have maximal nilpotency class.
By definition we have $[\gamma_1(S),\gamma_1(S)]=[\gamma_1(S),\gamma_2(S)] \leq \gamma_{4}(S)$ so $[\gamma_1(S) \colon [\gamma_1(S),\gamma_1(S)]] \geq p^3$. Therefore $\gamma_1(S)$ does not have maximal nilpotency class. Finally, if $M$ is a maximal subgroup of $S$ distinct from $\gamma_1(S)$ and $\C_S(\Z_2(S))$, then by Lemma  \ref{sbg.max.class} the group $M$ has maximal nilpotency class.
\end{proof}

We now prove some results concerning the sectional rank of $S$.

\begin{lemma}\label{lemma.1}
Suppose that $S$ has sectional rank $k\geq 2$ and order $|S|=p^{k+1}$. Then $S$ has a maximal subgroup $M$ that is elementary abelian. Furthermore, if $|S|\geq p^4$ then $M=\gamma_1(S)$.
\end{lemma}

\begin{proof}
If $|S|=p^3$ then $S$ is extraspecial and so contains a maximal subgroup that is elementary abelian.
Suppose $|S|\geq p^4$. By assumption there exists a subgroup $M\leq S$ having rank $k$.
So either $M=S$ or $[S \colon M] =p$ and $M$ is elementary abelian.
 If $M=S$ then, since $S$ has rank $2$, we deduce that $k=2$ and $|S|=p^3$, a contradiction. If
$M$ is a maximal subgroup of $S$, then $M$ contains $\gamma_2(S)$ and centralizes $\gamma_2(S)/\gamma_{4}(S)$ (since $M$ is abelian). Therefore $M=\gamma_1(S)$.
\end{proof}

\begin{lemma}\label{bound.k}
Suppose that $S$ has sectional rank $k\geq 2$ and order $|S|=p^n > p^{p+1}$. Then $k \leq p -1$.
\end{lemma}

\begin{proof}
Aiming for a contradiction, suppose there exists a $p$-group $S$ having maximal nilpotency class, order larger than $p^{p+1}$ and sectional rank $k\geq p$.
We may assume that $S$ is a minimal counterexample with respect to the order. Thus if $M$ is a $p$-group having maximal nilpotency class and order smaller than $|S|$ then either $M$ has sectional rank at most $p-1$ or $M$ has sectional rank $p$ and order $|M|= p^{p+1}$.

Since $|S|>p^{p+1}$, by Theorem \ref{power} we have
$ \gamma_i(S)^p = \gamma_{i+p-1}(S) \text{ for every } 1 \leq i \leq n-p+1.$
In particular $[\gamma_i(S) \colon \gamma_i(S)^p] \leq p^{p-1}$ for every $i$.

Let $P\leq S$ be a subgroup of $S$ having rank $k$. Since $|S|>p^{p+1}$ we have $\gamma_1(S)=\C_S(\Z_2(S))$ and by Lemma \ref{sbg.max.class} every maximal subgroup of $S$ distinct from $\gamma_1(S)$ has maximal nilpotency class.
Let $M$ be a maximal subgroup of $S$ containing $P$.

Suppose $M\neq \gamma_1(S)$. Then by the minimality of $S$ we deduce that $k=p$ and $|M|=p^{p+1}$. Hence by Lemma \ref{lemma.1} applied to $M$, either the group $M$ is extraspecial of order $p^3$ or $\gamma_2(S)$ ($=\gamma_1(M)$) is elementary abelian. In the first case $|S|=p^4 \leq p^{p+1}$ (since $p$ is odd) and we reach a contradiction. In the second case the group $\gamma_2(S)$ is elementary abelian of order $p^p$, which again is a contradiction ($[\gamma_2(S) \colon \gamma_2(S)^p] = p^{p-1}$).
Note that we proved that every maximal subgroup of $S$ distinct from $\gamma_1(S)$ has sectional rank smaller than $k$.

Therefore $M=\gamma_1(S)$ is the unique maximal subgroup of $S$ containing $P$.
 In particular $P\nleq \gamma_2(S)$ and so $\gamma_1(S)=P\gamma_2(S)$. Let $x\in \gamma_1(S) \backslash \gamma_2(S)$. Then $P=\langle x \rangle (\gamma_2(S) \cap P)$. Note that $\Phi(\gamma_2(S) \cap P) \leq \Phi(P)$ and $|(\gamma_2(S) \cap P) / \Phi(\gamma_2(S) \cap P)| < p^k$, since $\gamma_2(S)\cap P \leq \gamma_2(S) \leq N$ for some maximal subgroup $N$ of $S$ that is distinct from $\gamma_1(S)$.
Hence
 \[ p^k =[P \colon \Phi(P)] = p[P \cap \gamma_2(S) \colon \Phi(P)] \leq p[P\cap \gamma_2(S) \colon \Phi(P \cap \gamma_2(S))] < p^{k+1}.\]
Thus the only possibility is  $\Phi(P) = \Phi(\gamma_2(S) \cap P)$. In particular $\Phi(P) \leq \Phi(\gamma_2(S))$ and $x^p \in \Phi(\gamma_2(S))$.

Note that the group $\gamma_1(S)/\Phi(\gamma_2(S))$ has order at most $p^p$ ($[\gamma_2(S) \colon \Phi(\gamma_2(S))] \leq [\gamma_2(S) \colon \gamma_2(S)^p]=p^{p-1}$) and is therefore a regular $p$-group (as defined by P. Hall). Also notice that $\gamma_1(S)/\Phi(\gamma_2(S))=(\langle x \rangle \gamma_2(S))/\Phi(\gamma_2(S)) = \Omega_1(\gamma_1(S)/\Phi(\gamma_2(S)))$. So $\gamma_1(S)/\Phi(\gamma_2(S))$ has exponent $p$ and
$[\gamma_1(S) \colon \gamma_1(S)^p] \geq [\gamma_1(S) \colon \Phi(\gamma_2(S))] = p^p$, contradicting Theorem \ref{power}.

Therefore if $|S| > p^{p+1}$ then $k \leq p -1$.
\end{proof}

\begin{cor}\label{kleqp}
Suppose that $S$ has sectional rank $k\geq 2$. Then $p \geq k$, with equality only if $|S|=p^{k+1}$.
\end{cor}

\begin{proof}
If $|S| > p^{p+1}$ then by Lemma \ref{bound.k} we get $k \leq p-1$.
If $|S| \leq p^{p+1}$ then, since $|S|\geq p^{k+1}$, we deduce that $k+1 \leq p+1$ and so $k \leq p$.

Finally, if $p=k$ we get $p^{k+1} \leq |S| \leq p^{p+1}=p^{k+1}$ and so $|S|=p^{k+1}$.
\end{proof}

We now bound the order of $S$ by a function of its sectional rank.
Such bound will be improved once we add the assumption of a saturated fusion system defined on $S$ containing pearls.

\begin{theorem}\label{sectional.rank}
Suppose that $S$ has sectional rank $k\geq 2$ and $p\geq k+2$. Then $|S| \leq p^{2k}$, with strict inequality if $\gamma_1(S) = \C_S(\Z_2(S))$.
\end{theorem}

\begin{proof} Clearly the statement is true if $|S| =p^3$, so suppose $|S|\geq p^4$.
\begin{enumerate}
\item Assume $|S|=p^n \leq p^{p+1}$. Then for every $2\leq i \leq n$ the group $\gamma_i(S)$ has exponent $p$ by Theorem \ref{power}. By Theorem \ref{positive.deg} the degree of commutativity of $S$ is always non-negative, so  $[\gamma_{\lceil n/2 \rceil}(S), \gamma_{\lceil n/2 \rceil}(S)] \leq \gamma_{n}(S) = 1$. Therefore the subgroup $\gamma_{\lceil n/2 \rceil}(S)$ is elementary abelian and by definition of sectional rank we get ${\lfloor n/2 \rfloor} \leq k$ and so $n \leq 2k +1$.
In particular we get  $[\gamma_{k}(S), \gamma_{k}(S)]= [\gamma_{k}(S), \gamma_{k+1}(S)] \leq \gamma_{2k+1}(S)=1$, so $\gamma_{k}(S)$ is an elementary abelian subgroup of order $p^{n-k}$. Thus $n\leq 2k$.
Finally suppose $\gamma_1(S) = \C_S(\Z_2(S))$.
Then $S$ has positive degree of commutativity (Theorem \ref{positive.deg}) and so $[\gamma_{k-1}(S), \gamma_{k-1}(S)]=[\gamma_{k-1}(S),\gamma_{k}(S)] \leq \gamma_{2k}(S)$. If $|S|=p^{2k}$, then $\gamma_{k-1}(S)$ is an elementary abelian group of order $p^{k+1}$, contradicting the assumptions. Thus if $\gamma_1(S) = \C_S(\Z_2(S))$ then $|S|<p^{2k}$.

\item Assume $|S|=p^n > p^{p+1}$. By Theorem \ref{equal2step} we have $\gamma_1(S) = \C_S(\Z_2(S))$. Thus by Theorem \ref{positive.deg} we have $[\gamma_i(S),\gamma_j(S)] \leq \gamma_{i+j+1}(S)$ for every $1\leq i,j \leq n$.  In particular
\[ [\gamma_{k-1}(S), \gamma_{k-1}(S)]=[\gamma_{k-1}(S),\gamma_{k}(S)] \leq \gamma_{2k}(S). \]
Suppose for a contradiction that $n \geq 2k$. Then $[\gamma_{k-1}(S) \colon \gamma_{2k}(S)] = p^{k+1}$. Since $S$ has sectional rank $k$, we must have $\gamma_{2k}(S) < \Phi(\gamma_{k-1}(S)) = [\gamma_{k-1}(S), \gamma_{k-1}(S)]\gamma_{k-1}(S)^p$. Since $[\gamma_{k-1}(S),\gamma_{k-1}(S)]\leq \gamma_{2k}(S)$,  we must have  $\gamma_{2k}(S) < \gamma_{k-1}(S)^p$.
By Lemma \ref{power} either $\gamma_{k-1}(S)^p = 1$ or $\gamma_{k-1}(S)^p =\gamma_{k+p-2}(S)$. Therefore $\gamma_{2k}(S) < \gamma_{k+p-2}(S)$, implying that $k+ p -2 < 2k$. So
$p<k+2$, contradicting the assumption that $p\geq k+2$. Hence $|S| < p^{2k}$.
\end{enumerate}
\end{proof}

We conclude this section with a theorem about the automorphism group of a $p$-group of maximal nilpotency class.

\begin{theorem}\label{auto.S}
Suppose that $|S|=p^n\geq p^4$ and the group $\gamma_1(S)$ is neither abelian nor extraspecial.
Then $\Aut(S) \cong P \colon H$, where $P\in \Syl_p(\Aut(S))$ and $H$ is isomorphic to a subgroup of $\C_{p-1}$.
\end{theorem}

\begin{notation}
Let $P\leq S$ be a subgroup and let $u,v \in S$ be elements.
When we write $u \equiv v \mod P$ we mean that $u=vt$ for some $t \in P$. In particular the expression $u \equiv 1 \mod P$ is equivalent to $u\in P$.
\end{notation}

\begin{proof}
Since $S$ has maximal nilpotency class and order at least $p^4$, the quotient $S/\Phi(S)=S/\gamma_2(S)$ is elementary abelian of order $p^2$. In particular the group  $\Aut(S)/\C_{\Aut(S)}(S/\Phi(S))$ is isomorphic to a subgroup $H$ of $\GL_2(p)$.  Moreover, the fact that the subgroup $\gamma_1(S)$ of $S$ is characteristic in $S$ implies that for any $s_1 \in \gamma_1(S)\backslash \Phi(S)$ we can find a basis  $\mathcal{B}=\{ x\Phi(S), s_1\Phi(S) \}$ of $S/\Phi(S)$ such that every element $h$ of $H$ expressed with respect to $\mathcal{B}$ is of the form 
\[ h= \begin{pmatrix} a & 0 \\ c & b \end{pmatrix},\]
for some $a,b \in \GF(p)^*$ and $c \in \GF(p)$.
Since the group $\C_{\Aut(S)}(S/\Phi(S))$ is a $p$-group, we deduce that $\Aut(S)$ has a unique normal Sylow $p$-subgroup $P$.

Suppose there exists a morphism $\varphi \in \Aut(S)$ having order prime to $p$.
By Maschke's Theorem there exists a maximal subgroup $M/\Phi(S)$ of $S/\Phi(S)$ such that $M\neq \gamma_1(S)$ and $M$ is normalized by $\varphi$.  In other words we can find a basis $\mathcal{B}=\{ x\Phi(S), s_1\Phi(S) \}$ of $S/\Phi(S)$, with $s_1 \in \gamma_1(S) \backslash \Phi(S)$ and $x\in M \backslash \Phi(S)$, such that
\[ \varphi C =   \begin{pmatrix} a & 0 \\ 0 & b \end{pmatrix}C \quad \text{ with respect to the basis $\mathcal{B}$, }\]
for some $a,b\in \GF(p^*)$, where $C= \C_{\Aut(S)}(S/\Phi(S))$.

We prove that $b$ can be expressed as a power of $a$ (modulo $p$). This implies that $\Aut(S)/CP \cong \left\{ \begin{pmatrix} a & 0 \\ 0 & a^t \end{pmatrix} \mid a\in \GF(p)^*, t \in \mathbb{N} \right\} \leq \GL_2(p)$, proving the theorem.

Define
\[ s_i = [x, s_{i-1}] ~ \text{ for every } ~ 2\leq i \leq n-2 ~ \text{ and } \]
\[ s_{n-1} = \begin{cases} [x, s_{n-2}] \quad \text{ if } \gamma_1(S) =\C_S(\Z_2(S)) \\
						[s_1, s_{n-2}] \quad \text{otherwise}
\end{cases}\]

Since $x\notin \gamma_1(S)$ we deduce $\gamma_i(S) = \langle s_i \rangle \gamma_{i+1}(S)$ for every $2 \leq  i \leq n-1$.

The morphism $\varphi$ acts on every quotient $\gamma_i(S)/\gamma_{i+1}(S)$. We show by induction on $i$  that

\begin{equation}\label{action}
\begin{split}
 s_i \varphi & \equiv s_i^{a^{i-1}b} \mod \gamma_{i+1}(S) ~ \text{ for every } ~ 1 \leq i\leq n-2 ~ \text{ and } \\
s_{n-1}\varphi &= \begin{cases} s_{n-1}^{a^{n-2}b} \quad \text{ if } \gamma_1(S) =\C_S(\Z_2(S)) \\
						s_{n-1}^{a^{n-3}b^2} \quad \text{otherwise}
\end{cases}
\end{split}
\end{equation}

If $i=1$, then the identity (\ref{action}) is true by definition of $b$.
Assume $1<i<n-2$. Then by the inductive hypothesis we have
\[ s_i \varphi = [x, s_{i-1}]\varphi = [x^a u, s_{i-1}^{a^{i-2}b}v] \quad \text{ for some } \quad u\in \gamma_2(S), v\in \gamma_i(S). \]
Thus
\[    s_i \varphi \equiv [x^a, s_{i-1}^{a^{i-2}b}] \mod \gamma_{i+1}(S) \equiv s_i^{a^{i-1}b} \mod \gamma_{i+1}(S). \]

The same argument works for $i=n-1$ when $\gamma_1(S) =\C_S(\Z_2(S))$. 

If $\gamma_1(S)\neq \C_S(\Z_2(S))$ and ~$i=n-1$ then we have
\[ s_{n-1}\varphi = [s_1, s_{n-2}]\varphi = [s_1^b u, s_{n-2}^{a^{n-3}b}v] = s_{n-1}^{a^{n-3}b^2},\]
for some $u\in \gamma_2(S)$ and $v\in \Z(S)$.

We can now show that $b$ depends on $a$.
By assumption $\gamma_1(S)$ is non-abelian, so there exists $i < j \leq n-2$ such that $[s_i,s_j] \neq 1$. Then $[s_i,s_j] \in \gamma_r(S) \backslash \gamma_{r+1}(S)$ for some $ 1 \leq r \leq n-1$.
So $[s_i,s_j] = s_r^k \mod \gamma_{r+1}(S)$ for some $k \in \GF(p)$.

Suppose $\gamma_1(S) \neq \C_S(\Z_2(S))$. Then $\Z(S)=\gamma_{n-1}(S)=\Z(\gamma_1(S))$. Since $\gamma_1(S)$ is not extraspecial by assumption and $\gamma_1(S)^p\leq \Z(S)$ by Theorems \ref{equal2step}(3) and \ref{power}, we have  $\gamma_{n-1}(S) < [\gamma_1(S), \gamma_1(S)]$. Thus if $\gamma_1(S) \neq \C_S(\Z_2(S))$ we may assume $r< n- 1$.

Returning to the general case, the identity (\ref{action}) implies
\[s_r^k \varphi = (s_r\varphi)^k \equiv s_r^{k(a^{r-1}b)} \mod \gamma_{r+1}(S).\]
On the other hand,
\[ s_r^k \varphi = [s_i,s_j]\varphi = [s_i\varphi, s_j\varphi] \equiv [s_i^{a^{i-1}b}, s_j^{a^{j-1}b}] \mod \gamma_{r+1}(S) \equiv (s_r^k)^{a^{i+j-2}b^2} \mod \gamma_{r+1}(S).\]
Hence
\[ b \equiv a^{r+1-i-j} \mod p. \]
\end{proof}

\begin{remark}
If the group $\gamma_1(S)$ is extraspecial, then the conclusion of Theorem \ref{auto.S} is not true. As an example, if $S$ is a Sylow $p$-subgroup of the group $\G_2(p)$, then $S$ has maximal nilpotency class, the subgroup $\gamma_1(S)$ of $S$ is extraspecial and $|\Aut(S)|$ is divisible by $(p-1)^2$.
\end{remark}

\section{Properties of pearls}
Let $p$ be an odd prime, let $S$ be a $p$-group and let $\F$ be a saturated fusion system on $S$.
Suppose $E\in \mathcal{P}(\F)$ is a pearl. By the definition of pearl and the fact that $E$ is $\F$-centric we deduce that either $\Phi(E)=1$ or $\Phi(E)=\Z(E) =\Z(S)$ and in both cases the quotient $E/\Phi(E)$ is a self-centralizing subgroup of $S/\Phi(E)$ having index $p$ in its normalizer. In other words, the group $E/\Phi(E)$ is a soft subgroup of $S/\Phi(E)$, as defined by H\'{e}thelyi.

\begin{definition}
A subgroup $A$ of a $p$-group $P$ is said to be \emph{soft} in $P$ if $\C_P(A)=A$ and  $[\N_P(A) \colon A] =p$.
\end{definition}

\begin{definition}
Let $Q$ be a subgroup of the $p$-group $P$ such that $[P \colon Q] =p^m$.
We set $\N^0(Q) = Q$ and $\N^i(Q)=\N_P(\N^{i-1}(Q))$ for every $1\leq  i \leq m$.
The sequence $\N^0(Q) < \N^1(Q) < \dots < \N^m(Q)=P$ is called the \emph{normalizer tower} of $Q$ in $P$.
If $[\N^i(Q) \colon \N^{i-1}(Q)] =p$ for every $1\leq  i \leq m$ then we say that $Q$ has \emph{maximal} normalizer tower in $P$.
\end{definition}

The next theorem describes some properties of soft subgroups, that are in particular satisfied by the subgroup $E/\Phi(E)$ of $S/\Phi(E)$ whenever $E$ is a pearl of $S$.

\begin{theorem}\label{Het} \cite[Lemma 2, Corollary 3]{Het1}\cite[Theorem 1, Lemma 1 and Corollary 6]{Het2}\cite[Theorem 2.1]{Het3}
Let $P$ be a $p$-group and let $A$ be a soft subgroup of $P$ such that $[P \colon A]= p^m\geq p^2$. Set
\[H_i= \begin{cases} \Z_i(\N^i(A)) \quad &\text{ if }\quad 1\leq i \leq m-1
\\ \Z(\N^1(A))[P,P]  \quad &\text{ if }\quad i=m \end{cases}\]
Then

\begin{enumerate}\setlength\itemsep{0.5em}
\item $A$ has maximal normalizer tower in $P$ and the members of such a tower are the only subgroups of $P$ containing $A$;
\item the group $\N^i(A)$ has nilpotency class $i+1$ for every $i\leq m-1$;
\item $H_i \leq \N^{i-1}(A)$ and $H_i$ is characteristic in $\N^i(A)$ for every $1\leq i \leq m$;
\item $[H_{i+1} \colon H_i] = [\N^{i}(A) \colon H_{i+1}] = [\N^0(A) \colon H_1]= p$  for every $1\leq i\leq m-1$;
\item $\N^i(A)/H_i \cong \C_p \times \C_p$  for every $1\leq i \leq m$;
\item the members of the sequence \\$\Z(\N^1(A))=H_1 < H_2 < \dots < H_{m-1} < H_m$ \\are the only subgroups of $H_m$ normalized by $A$ that contain $\Z(\N^1(A))$;
\item if $Q$ is a soft subgroup of $P$ with $[P \colon Q]  \geq p^2$, then $H_m =  \Z(\N^1(Q))[P,P]$;
\item if $Q$ is a soft subgroup of $P$ and $Q\leq \N^{m-1}(A)$ then there exists $g\in P$ such that $Q^g=A$.
\end{enumerate}
\end{theorem}

We can now start to study properties of pearls.

\begin{lemma}\label{pearl}
Suppose $S$ has maximal nilpotency class and order $|S|\geq p^4$ and let $E\leq S$ be an $\F$-essential subgroup of $S$. Then the following are equivalent:
\begin{enumerate}
\item $E$ is a pearl;
\item $E$ is contained in neither $\gamma_1(S)$ nor $\C_S(\Z_2(S))$;
\item there exists an element $x\in S\backslash \C_S(\Z_2(S))$ of order $p$ such that \[ \text{ either } \quad E=\langle x\rangle \Z(S) \quad \text{ or } \quad E=\langle x \rangle \Z_2(S);\]
\end{enumerate}
\end{lemma}

\begin{remark}
Let $\E$ be a saturated fusion system on a $p$-group $P$ containing a unique index $p$ abelian subgroup $A$ and suppose that there are $\E$-essential subgroups of $P$ distinct from $A$. Then by \cite[Lemma 2.3(a)]{p.index} the $\E$-essential subgroups of $P$ distinct from $A$ are of the form $\langle x \rangle \Z(P)$ or $\langle x \rangle\Z_2(P)$ for some $x\in P \backslash A$. Also, if $|\Z(P)|=p$ then $P$ has maximal nilpotency class \cite[Lemma 2.3(b)]{p.index} and by Lemma \ref{pearl} we conclude that the $\E$-essential subgroups of $P$ distinct from $A$ are pearls.
\end{remark}

\begin{proof}
From the assumption $|S|\geq p^4$ we deduce that $\Z_2(S) < \C_S(\Z_2(S))$. In particular $\Z_2(S)$ is not $\F$-centric and so it is not a pearl.

\begin{description}\setlength\itemsep{0.5em}

\item[$(1 \Rightarrow 2)$]
Suppose $E$ is a pearl. Then by Lemma \ref{self-centr.max.class} and  Corollary \ref{max.sbg.no.max.class} we conclude that the pearl $E$ is contained in neither $\gamma_1(S)$ nor $\C_S(\Z_2(S))$.

\item[$(2 \Rightarrow 3)$]
Suppose $E$ is contained in neither $\gamma_1(S)$ nor $\C_S(\Z_2(S))$.
By Lemma \ref{sbg.max.class}, we get that $E$ has maximal nilpotency class and so $E$ is a pearl by Corollary \ref{max.class.not.ess}.
Lemma \ref{sbg.max.class} also tells us that
if $|E|=p^m$ then $[E \colon Z_{m-1}(S)] =p$. Thus there exists an element $x\in S$ such that either $E=\langle x\rangle \Z(S)\cong \C_p \times \C_p$ or $E=\langle x \rangle \Z_2(S) \cong p^{1+2}_+$. Note in particular that $x\notin \C_S(\Z_2(S))$ because $E\nleq \C_S(\Z_2(S))$.

\item[$(3 \Rightarrow 1)$] Suppose statement 3 holds. Since $S$ has maximal nilpotency class we have $\Z(S) \cong \C_p$ and $\Z_2(S) \cong \C_p \times \C_p$. Thus either $E$ has order $p^2$ or $E$ is non-abelian of order $p^3$. Therefore $E$ is a pearl.  
\end{description}
\end{proof}

\begin{theorem}\label{lift}
Let $E\in \mathcal{P}(\F)$ be a pearl with $[S \colon E]=p^m$. Then $E$ has maximal normalizer tower in $S$, the members $\N^i(E)$ of such tower are the only subgroups of $S$ containing $E$ and for every $1 \leq i\leq m-1$  the group $\N^i(E)$ is not $\F$-essential and $\Aut_S(\N^i(E)) \norm \Aut_\F(\N^i(E))$.
In particular every morphism in $\N_{\Aut_\F(E)}(\Aut_S(E))$ is the restriction of a morphism in $\Aut_\F(S)$ that normalizes each member of the normalizer tower.
\end{theorem}

\begin{proof}
As we noticed before the group $E/\Phi(E)$ is soft in $S/\Phi(E)$. Also, $\N^i(E/\Phi(E)) = \N^i(E)/\Phi(E)$.
Hence by Theorem \ref{Het} the group $E$ has maximal normalizer tower in $S$ and the members of such tower are the only subgroups of $S$ containing $E$.
By  Lemma \ref{self-centr.max.class} for every $1\leq  i \leq m$ the group $\N^i(E)$ has maximal nilpotency class. In particular if either $i\geq 2$ or $i\geq 1$ and $E$ is extraspecial,  then $\N^i(E)$ is not $\F$-essential by Corollary \ref{max.class.not.ess}.

Note that for every $i\geq 1$ the group $\N^i(E)$ is $\F$-centric, since it contains $E$.
Also, for every $i\geq 1$ and every $\alpha \in \Hom_\F(\N^i(E),S)$ the group $E\alpha$ is an $\F$-pearl by Corollary \ref{auto.pearl}. In particular $E\alpha$ has maximal normalizer tower in $S$. Since $\F$ is saturated, this fact is enough to guarantee that for every $i\geq 1$ the group $\N^i(E)$ is fully normalized in $\F$. In particular $\Aut_S(\N^i(E))\in \Syl_p(\Aut_\F(\N^i(E)))$ for every $i\geq 1$.
By the definition of $\F$-essential subgroup, we deduce that if either $i\geq 2$ or $i\geq 1$ and $E$ is extraspecial, then the group $\Out_\F(\N^i(E))$ does not have a strongly $p$-embedded subgroup. Since $E$ has maximal normalizer tower in $S$, we have $|\Out_S(\N^i(E))| = p$ for $1\leq i \leq  m-1$. Thus if either $i\geq 2$ or $i\geq 1$ and $E$ is extraspecial,  then we must have $\Out_S(\N^i(E))\norm \Out_\F(\N^i(E))$ by \cite[Proposition A.7(b)]{AKO} and so $\Aut_S(\N^i(E)) \norm \Aut_\F(\N^i(E))$.

We now show that if $E$ is abelian then $\N^1(E)$ is not $\F$-essential and $\Aut_S(\N^1(E)) \norm \Aut_\F(\N^1(E))$. If $\N^1(E)=S$ then there is nothing to prove. Suppose $\N^1(E) < S$. Note that $\N^1(E)\cong p^{1+2}_+$. In particular  $\N^1(E)$  has $p+1$ maximal subgroups and at least $p$ of them are conjugated to $E$ in $\N^2(E)$. The group $\Phi(\N^2(E))$ is a maximal subgroup of $\N^1(E)$, since $[\N^2(E) \colon \N^1(E)]=p$ and $\N^2(E)$ has maximal nilpotency class and so rank $2$. However $\Phi(\N^2(E))$ cannot be $\F$-conjugated to $E$ because $\Phi(\N^2(E)) \norm \N^2(E)$ and $E$ is fully normalized in $\F$. Hence $\Phi(\N^2(E))$ is normalized by  $\Aut_\F(\N^1(E))$. Since $\Aut_S(\N^1(E))$ stabilizes the sequence of subgroups $\Phi(\N^1(E)) < \Phi(\N^{2}(E)) < \N^1(E)$, by\cite[Corollary 5.3.3]{Gor}  we deduce that $\Aut_S(\N^1(E))\norm \Aut_\F(\N^1(E))$. In particular $\Out_\F(\N^1(E))$ does not contain strongly $p$-embedded subgroups and so $\N^1(E)$ is not $\F$-essential.

We finally show that every morphism in $\N_{\Aut_\F(E)}(\Aut_S(E))$ is the restriction of a morphism in $\Aut_\F(S)$ that normalizes each member of the normalizer tower.
Since $E$ is fully normalized in $\F$, it is receptive in $\F$ (\cite[Proposition 1.10]{AO}) and so every morphism $\varphi \in \Aut_\F(E)$ is the restriction of a morphism $\varphi_1$ in $\Hom_\F(\N_\varphi, S)$, where $\N_\varphi = \{ x \in \N^1(E) \mid \varphi c_x \varphi^{-1} \in \Aut_S(E) \}$ (\cite[Definition 1.2]{AO}). If $\varphi \in \N_{\Aut_\F(E)}(\Aut_S(E))$ then $\N_\varphi = \N^1(E)$ by definition. Since the morphism $\varphi_1$ normalizes $E$,
we also deduce that  $\N^1(E)\varphi_1=\N^1(E)$ and so $\varphi_1 \in \Aut_\F(\N^1(E))$. Note that  $\N^1(E)$ is fully normalized, and so receptive, and  we showed that $\Aut_S(\N^1(E)) \norm \Aut_\F(\N^1(E))$. Hence $\N_{\varphi_1}=\N^2(E)$ and we can repeat the argument. Such iteration works for every $1\leq  i \leq m-1$.  In other words, for every $1 \leq i \leq m-1$ there exists an automorphism  $\varphi_{i+1} \in \Aut_\F(\N^{i+1}(E))$ such that $\varphi_{i+1}|_E = \varphi$ and $\varphi_{i+1}|_{\N^j(E)} =\varphi_j$ for every $1\leq j \leq i$. 
Therefore every morphism in $\N_{\Aut_\F(E)}(\Aut_S(E))$ is the restriction of a morphism in $\Aut_\F(S)$ that normalizes each member of the normalizer tower.
\end{proof}

If $E\in \mathcal{P}(\F)$ then we can construct saturated fusion subsystems $\E_i$ of $\F$ defined on $\N^i(E)$ and such that $E \in \mathcal{P}(\E_i)$ for every $1\leq i \leq m-1$.

\begin{lemma}\label{inclusion}
Let $E \in \mathcal{P}(\F)$ be a pearl and let $m$ be such that $[S \colon E] =p^m$. For every $1 \leq i \leq m-1$ let $\E_i$ be the smallest fusion subsystem of $\F$ defined on $\N^i(E)$ such that  $\Aut_{\E_i}(E)=\Aut_\F(E)$ and $\Aut_{\E_i}(\N^i(E)) = \Inn(\N^i(E))\N_{\Aut_\F(\N^i(E))}(E)$. Then $\E_i$ is a saturated fusion subsystem of $\F$ and $E\in \mathcal{P}(\E_i)$. 
\end{lemma}

\begin{remark}\label{auto.remark} Note that for every $1 \leq i \leq m-1$ we have \[\Aut_\F(\N^i(E)) = \Aut_S(\N^i(E))\N_{\Aut_\F(\N^i(E))}(E).\]
Indeed for every $i\geq 1$ the group $\N^i(E)$ has maximal nilpotency class by Lemma \ref{self-centr.max.class} and the group $\Aut_\F(\N^i(E))$ acts on $\N^i(E)/\Phi(\N^i(E)) \cong \C_p \times \C_p$. By Theorem \ref{lift} the pearl $E$ has maximal normalizer tower in $S$ and for every $i\geq 1$ we have $\Aut_S(\N^i(E))\norm \Aut_\F(\N^i(E))$. In particular for every $1 \leq  i \leq m-1$ the group $\Aut_S(\N^i(E))$ acts transitively on the conjugates of $\N^{i-1}(E)$ contained in $\N^i(E)$ and by the Frattini Argument (\cite[3.1.4]{KS}) we get $\Aut_\F(\N^i(E)) = \Aut_S(\N^i(E))\N_{\Aut_\F(\N^i(E))}(\N^{i-1}(E))$. Since this statement is true for every $1\leq  i \leq m-1$, we conclude that $\Aut_\F(\N^i(E)) = \Aut_S(\N^i(E))\N_{\Aut_\F(\N^i(E))}(E)$.
\end{remark}

\begin{proof}
Fix $1\leq i \leq m-1$ and set $\N=\N^i(E)$ and $\E=\E_i$. We only need to prove that  $\E$ is saturated. 
Set $G= \N \colon \N_{\Out_\F(\N)}(E)$, $\Delta= \Aut_\F(E)$ and $K=\Out_G(E) = \Out_S(E)\N_{\Out_\F(\N)}(E)$. Hence $G$, $E$, $\Delta$ and $K$ satisfy the assumptions of \cite[Proposition 5.1]{BLO}. By definition, $\E$ is the smallest fusion subsystem of $\F$ containing $\F_\N(G)$ and $\Delta$, and so $\E$ is saturated.  
\end{proof}

\begin{lemma}\label{max.sbg}
Let $E,P \in \mathcal{P}(\F)$ be pearls and suppose that $P\notin E^\F$. Let $M_E$ and $M_P$ be the unique maximal subgroups of $S$ containing $E$ and $P$, respectively (whose uniqueness is guaranteed by Theorem \ref{lift}). Then $M_P \notin M_E^\F$ (in particular $M_P\neq M_E$).
\end{lemma}

\begin{proof}
Assume by contradiction that $M_E = M_P\alpha$ for some $\alpha \in \Hom_\F(M_P,S)$. By Corollary \ref{auto.pearl} the group $P\alpha$ is a pearl. Also, $P\alpha \leq M_P\alpha= M_E$ and $P\alpha \notin E^\F$. Upon replacing $P$ by $P\alpha$ we can assume that $M_E = M_P$.

Suppose $E\cong P$. Then $\Phi(E)=\Phi(P)$ and by Theorem \ref{Het}(8) the groups $E$ and $P$ are conjugated in $S$, contradicting the fact that $P\notin E^\F$. Thus $E$ and $P$ are not isomorphic and we can assume that $E$ is abelian and $P$ is extraspecial. By Lemma \ref{lift} the group $\N^1(E)$ is not $\F$-essential so $\N^1(E)\neq P$.
Let $i\geq 2$ be the smallest integer such that $P\leq \N^i(E)$ (such $i$ exists because $P\leq M_E$). The maximal subgroups of $\N^i(E)$ are $\Z_{i+1}(S)$, $\N^{i-1}(E)$ and $p-1$ subgroups conjugated in $S$ to $\N^{i-1}(E)$. Therefore there exists $g\in S$ such that $P^g \leq \N^{i-1}(E)$. Iterating this argument we get that there exists $g\in S$ such that $P^g = \N^1(E)$, contradicting the fact that $P$ is $\F$-essential.

Therefore $M_E \neq M_P\alpha$ for every $\alpha \in \Hom_\F(M_P,S)$ and $M_P \notin M_E^\F$.
\end{proof}

\begin{definition}\label{delta}
Let $E\in \mathcal{P}(\F)$ be a pearl. Let  $\lambda\in \GF(p)^*$ be an element of order $p-1$. We denote by $\varphi_\lambda$ the automorphism of $S$ that normalizes $E$ and acts on $E/\Phi(E)$ as \[\begin{pmatrix} \lambda^{-1} & 0 \\ 0 & \lambda \end{pmatrix},\]
with respect to some basis $\{ x\Phi(E), z\Phi(E)\}$ of $E/\Phi(E)$, where $\langle z \rangle= \Z(S)$ if $E$ is abelian and $\langle z \rangle\Phi(E)=\Z_2(S)$ otherwise.

We set $\Delta_\F(E) = \{ \varphi_\lambda \in \Aut_\F(S)| \lambda\in \GF(p)^*$ is an element of order $p-1 \}$.
\end{definition}

Note that $\Delta_\F(E)$ is a subset of $\Aut_\F(S)$ and Corollary \ref{auto.pearl} and Theorem \ref{lift} guarantee that $\Delta_\F(E) \neq \emptyset$. Each automorphism $\varphi_\lambda \in \Delta_\F(E)$ normalizes the members of the normalizer tower of $E$ and the members of the lower central series of $S$. The study of the action of the elements of $\Delta_\F(E)$ on the subgroups of $S$ is the key point of the proofs of the main theorems appearing in this paper.

\begin{lemma}\label{action.frat}
Let $E\in \mathcal{P}(\F)$ be a pearl and let $\varphi_\lambda\in \Delta_\F(E)$. Then $\varphi_\lambda$ centralizes $\Phi(E)$.
\end{lemma}

\begin{proof}
If $E$ is abelian then $\Phi(E)=1$ and there is nothing to prove. Suppose $E$ is extraspecial and $\Phi(E)=\Z(S)$. Let $x,z\in E\backslash \Phi(E)$ be such that $x\varphi_\lambda = x^{\lambda^{-1}}u$ and $z\varphi_\lambda = z^{\lambda}v$, for some $u,v \in \Phi(E)=\Z(S)$. By properties of commutators (\cite[Lemma 2.2.2]{Gor}) and the fact that $[x,z] \in[E,E] = \Phi(E) = \Z(S)$ we get
\[ [x,z] \varphi_\lambda = [x\varphi_\lambda, z\varphi_\lambda] =[x^{\lambda^{-1}}u,   z^{\lambda}v] = [x,z].\]
Note that $\langle [x,z] \rangle = \Phi(E)$ because $E$ is not abelian. Therefore $\Phi(E)$ is centralized by $\varphi_\lambda$.
\end{proof}

\begin{lemma}\label{lambda.action} Suppose $|S|=p^n \geq p^4$. Let $E\in \mathcal{P}(\F)$ be a pearl and let $\varphi_\lambda\in \Delta_\F(E)$.
Then for every $1\leq i \leq n-1$, and every $s_{i} \in \gamma_i(S) \backslash \gamma_{i+1}(S)$ we have
\[ s_i\varphi \equiv s_i^{a_i} \mod \gamma_{i}(S)^p \quad  \text{ with } \quad a_i = \lambda^{n-i-\epsilon},\]
where $\epsilon = 0$ if $E$ is abelian and $\epsilon =1$ otherwise.
\end{lemma}

\begin{proof}
Let $x \in E$ be such that $x$ is contained in neither $\gamma_1(S)$ nor $\C_S(\Z_2(S))$ and $x\varphi_{\lambda} = x^{\lambda^{-1}}u$ for some $u \in \Phi(E)\leq \Z(S)$.
Note that the existence of $x$ is guaranteed by Lemma \ref{pearl} and the existence of $\varphi_\lambda$.
Let $s_1$ be an element of $\gamma_1(S)$ not contained in $\gamma_2(S)$ and set $s_i= [x, s_{i-1}]$ for every $i\geq 2$. Since $x\notin \gamma_1(S)$ and $x\notin \C_S(\Z_2(S))$ we deduce that $\gamma_i(S) = \langle s_i \rangle \gamma_{i+1}(S)$ for every $i\geq 1$.

Set $\epsilon = 0$ if $E$ is abelian and $\epsilon=1$ otherwise.
We first prove by downward induction on $i\leq n-1$ that \[s_i\varphi \equiv s_i^{a_i} \mod \gamma_{i+1}(S),  \text{ where } a_i = \lambda^{n-i-\epsilon}.\]

Suppose $i=n-1$. Note that $s_{n-1} \in \Z(S)$ and by definition of $\varphi_\lambda$ if $E$ is abelian and Lemma \ref{action.frat} otherwise, we get \[ s_{n-1}\varphi_{\lambda} = s_{n-1}^{a_{n-1}} \quad \text{ where } a_{n-1} = \lambda^{1-\epsilon}.\]
Suppose the statement is true for $2\leq i+1\leq n-1$.  Hence
\begin{equation}\label{eq1} s_{i+1}\varphi_\lambda \equiv s_{i+1}^{\lambda^{n-(i+1)-\epsilon}} \mod \gamma_{i+2}(S).\end{equation}

Note that $s_i\varphi_{\lambda} = s_i^{a_i}h$ for some $a_i\in \GF(p)^*$ and $h\in \gamma_{i+1}(S)$.
We have $[x,h]\in \gamma_{i+2}(S)$ and $\gamma_i(S)/\gamma_{i+2}(S) \cong \C_p \times \C_p$.
Therefore
\begin{equation}\label{eq2} [x,s_i] \varphi_{\lambda} = [x^{\lambda^{-1}}u, s_i^{a_i}h] \equiv [x,s_i]^{\lambda^{-1}a_i} \equiv s_{i+1}^{\lambda^{-1}a_i} \mod \gamma_{i+2}(S).\end{equation}

On the other hand, $[x,s_i] = s_{i+1}$, so comparing equations (\ref{eq1}) and (\ref{eq2})  we get
\[s_{i+1}^{\lambda^{n-(i+1)-\epsilon}} \equiv s_{i+1}^{\lambda^{-1}a_i} \mod \gamma_{i+2}(S).\]
Hence
\[ s_i\varphi_\lambda \equiv s_i^{a_i} \mod \gamma_{i+1}(S) \quad  \text{ with } \quad a_i = \lambda^{n-i-\epsilon}.\]

It remains to show that we have indeed $s_i\varphi_\lambda \equiv s_i^{a_i} \mod \gamma_i(S)^p$.

If $i=n-1$ then the statement is true.
Suppose $i< n-1$. Then $\gamma_i(S)/\gamma_{i+2}(S)\cong \C_p \times \C_p$ and $\varphi_\lambda$ normalizes $\gamma_{i+1}(S)/\gamma_{i+2}(S)$. By Maschke's Theorem there exists a subgroup $\gamma_{i+2}(S) < V_i < \gamma_i(S)$ such that $V_i\neq \gamma_{i+1}(S)$ and $V_i/\gamma_{i+2}(S)$ is normalized by $\varphi_{\lambda}$. More precisely the action of $\varphi_{\lambda}$ on $V_i/\gamma_{i+2}(S)$ is the same as the action on $\gamma_i(S)/\gamma_{i+1}(S)$. Thus we can assume $s_i\varphi_\lambda \equiv s_i^{a_i} \mod \gamma_{i+2}(S)$.
If $\gamma_{i+2}(S) =\gamma_i(S)^p$ then we are done. Suppose $\gamma_i(S)^p < \gamma_{i+2}(S)$. Since the quotient $\gamma_i(S)/\gamma_i(S)^p$ has exponent $p$, we deduce that $V_i/\gamma_{i+3}(S)\cong \C_p \times \C_p$. Hence by Maschke's Theorem there exists $\gamma_{i+3}(S)< V_{i+1} < V_i$ normalized by $\varphi_\lambda$ and distinct from $\gamma_{i+2}(S)$. Thus $s_i\varphi_\lambda \equiv s_i^{a_i} \mod \gamma_{i+3}(S)$.
We can iterate this process until we get  $s_i\varphi_\lambda \equiv s_i^{a_i} \mod \gamma_i(S)^p$.
\end{proof}

The next result is a first application of Lemma \ref{lambda.action}.

\begin{lemma}\label{exp.Si}
Suppose that $p$ is an odd prime, $S$ is a $p$-group having sectional rank $k$ and order at least $p^4$ and $\F$ is a saturated fusion system on $S$ such that $\mathcal{P}(\F) \neq \emptyset$ (so $S$ has maximal nilpotency class by Lemma \ref{self-centr.max.class}). Then for every $i\geq 1$, if $\gamma_i(S)$ has order at most $p^{p-1}$ then $\gamma_i(S)$ has exponent $p$. In particular if $|S|\leq p^{p}$ then $S$ has exponent $p$.
\end{lemma}

\begin{proof}Let $E \in \mathcal{P}(\F)$ be a pearl, let $\varphi_\lambda\in \Delta_\F(E)$ and let $\epsilon=0$ if $E$ is abelian and $\epsilon=1$ otherwise.
Aiming for a contradiction, suppose there exists $i\geq 1$ such that the group $\gamma_i(S)$ has order at most $p^{p-1}$ and does not have exponent $p$. By Lemma \ref{power} we get $|S|=p^n\leq p^p$, $\gamma_i(S) = \gamma_1(S)$  and $\gamma_1(S)^p = \Z(S)$.
Take $x\in \gamma_1(S) \backslash \gamma_2(S)$. Then by Lemma \ref{lambda.action} we get
\[ (x^p)^{\lambda^{1 - \epsilon}} = (x^p)\varphi_\lambda = (x\varphi_\lambda)^p =(x^{\lambda^{n-1-\epsilon}}z)^p,\]
for some $z\in \Z(S)=\gamma_1(S)^p$.
Note that $z^p=1$ and since $z$ commutes with $x$ we deduce $(x^p)^{\lambda^{1 - \epsilon}} = (x^{\lambda^{n-1-\epsilon}})^p$. Note that $\gamma_1(S)$ is a regular group because it has order at most $p^{p-1}$. If $x^p=1$ then $\gamma_1(S)$ is generated by elements of order $p$ ($\gamma_2(S)$ has exponent $p$ because $|S|\leq p^p$) and so it has exponent $p$, contradicting the assumptions. Thus $x^p\neq 1$ and $1 \equiv n-1 \mod p-1$, that implies $n\equiv2 \mod p-1$. However $3 \leq n\leq p$,  a contradiction. Thus the group $\gamma_i(S)$ has exponent $p$.

Note that $E\nleq \gamma_1(S)$ by Lemma \ref{pearl}, so $S=E\gamma_1(S)$. If $|S|\leq p^p$ then $|\gamma_1(S)|\leq p^{p-1}$ and $S$ is a regular group (as defined by P. Hall) generated by elements of order $p$. Therefore $S$ has exponent $p$.
\end{proof}

We proceed characterizing the case in which the group $\gamma_1(S)$ is extraspecial.

\begin{theorem}\label{S1extraspecial}
Let $p$ be an odd prime, let $S$ be a $p$-group and let $\F$ be a saturated fusion system on $S$. If $\mathcal{P}(\F)\neq \emptyset$ then $S$ has maximal nilpotency class and the following are equivalent:
\begin{enumerate}
\item $\gamma_1(S)$  is extraspecial;
\item $\gamma_1(S) \neq \C_S(\Z_2(S))$;
\item $\mathcal{P}(\F)_a \neq \emptyset, |S| = p^{p-1}$ and $\gamma_1(S)$ is not abelian.
\end{enumerate}
In particular, if any of the above conditions holds, then $p\geq 7$ and $S$ has exponent $p$.
\end{theorem}

\begin{proof}
By Lemma \ref{self-centr.max.class} the group $S$ has maximal nilpotency class. 
Let $E \in \mathcal{P}(\F)$ be a pearl and  let $\varphi_\lambda\in \Delta_\F(E)$.
\begin{description}\setlength\itemsep{0.5em}
\item[$(1 \Rightarrow 2)$] Suppose $\gamma_1(S)$ is extraspecial. Then $\Z(\gamma_1(S))=\Z(S)$ so $\gamma_1(S) \neq \C_S(\Z_2(S))$.

\item[$(2 \Rightarrow 3)$]  Suppose $\gamma_1(S) \neq \C_S(\Z_2(S))$.
Since $\gamma_1(S)$ does not centralize $\Z_2(S)$, it cannot be abelian.
Let $M$ be the unique maximal subgroup of $S$ containing $E$ (whose uniqueness is guaranteed by Theorem \ref{Het}(1)). Then $M\neq \gamma_1(S)$ and $M\neq \C_S(\Z_2)$ by Lemma \ref{pearl}.
So the morphism $\varphi_\lambda$ normalizes the distinct groups $\gamma_1(S)$, $\C_S(\Z_2(S))$ and $M$. Since $S/\gamma_2(S) \cong \C_p \times \C_p$ we deduce that $\varphi_\lambda$ acts as a scalar on $S/\gamma_2(S)$.
If $E\cong p^{1+2}_+$ then by Lemma \ref{lambda.action} the morphism $\varphi_\lambda$ acts on $\gamma_1(S)/\gamma_2(S)$ as $\lambda^{n-2}$. Thus we need $n-2 \equiv -1 \mod (p-1)$, that is $n \equiv1 \mod(p-1)$. In particular $n$ is odd, as $p-1$ is even. Therefore by Theorem \ref{equal2step} we conclude $\gamma_1(S)=\C_S(\Z_2(S))$, contradicting the assumptions. Hence we need $E\cong \C_p \times \C_p$. In this case, the morphism $\varphi_\lambda$ acts on $\gamma_1(S)/\gamma_2(S)$ as $\lambda^{n-1}$. So $n-1 \equiv -1 \mod (p-1)$, that is $n \equiv0 \mod (p-1)$. So the group $S$ has order $p^{\alpha(p-1)}$ for some $\alpha \in \mathbb{N}$. By Theorem \ref{equal2step} we also have $6 \leq \alpha(p-1) \leq p+1$. Thus $\alpha =1$ and $|S|=p^{p-1}$.

\item[$(3 \Rightarrow 1)$]  Suppose $\mathcal{P}(\F)_a \neq \emptyset$, $|S|=p^{p-1}$ and $\gamma_1(S)$ is not abelian.
Assume $E \in \mathcal{P}(\F)_a$ and let $s_i\in \gamma_i(S) \backslash \gamma_{i+1}(S)$ for $1\leq i \leq n-1$. We first prove that $[s_1,s_i]=1$ for every $1 \leq i \leq n-3$. Since $|S|=p^{p-1}$, by Lemma \ref{exp.Si} we have $\gamma_i(S)^p=1$ for every $i\geq 1$.
Then by Lemma \ref{lambda.action} we get $s_i\varphi_\lambda = s_i^{\lambda^{-i}}$ for every $i\geq 1$.
Fix $2 \leq i\leq n-3$ and aiming for a contradiction suppose that $[s_1,s_i] \neq 1$. Note that $[\gamma_1(S),\gamma_i(S)] \leq \gamma_{i+2}(S)$. Hence $[s_1,s_i] \in \gamma_{k}(S) \backslash \gamma_{k+1}(S)$ for some $i+2 \leq k \leq p-2$.
By properties of commutators we get
\[ [s_1,s_i]^{\lambda^{-k}} = [s_1,s_i]\varphi_\lambda = [s_1^{\lambda^{-1}}, s_i^{\lambda^{-i}}] = [s_1,s_i]^{\lambda^{-1-i}}.\]
Hence $k\equiv 1+i \mod p-1$, contradicting the fact that $i+2 \leq k \leq p-2$. Therefore $[s_1,s_i]=1$ for every $1 \leq i \leq n-3$.

If $\gamma_1(S) =\C_S(\Z_2(S))$ then the same argument shows that $[s_1,s_{n-2}] = 1$ and so  $s_1 \in \Z(\gamma_1(S))$. Since $\Z(\gamma_1(S))=\gamma_i(S)$ for some $i$ and $s_1\notin \gamma_2(S)$ we get $\gamma_1(S)=\Z(\gamma_1(S))$, contradicting the assumption that $\gamma_1(S)$ is not abelian. Hence $\gamma_1(S) \neq \C_S(\Z_2(S))$ and so $\Z(\gamma_1(S))=\Z(S)$.

Consider the group $S/\Z(S)$, which has maximal nilpotency class. Let $Z$ be the preimage in $S$ of $\Z(\gamma_1(S)/\Z(S))$. Then $Z \leq \gamma_1(S)$ and $Z\norm S$, so $Z=\gamma_i(S)$ for some $i$. Also, $s_1\in Z$ by what we proved above. Since $s_1\notin \gamma_2(S)$ we conclude $Z=\gamma_1(S)$. Therefore the group $\gamma_1(S)/\Z(S)$ is abelian. Hence $[\gamma_1(S),\gamma_1(S)] = \Z(S) = \Z(\gamma_1(S))$ and since $\gamma_1(S)$ has exponent $p$ we get $[\gamma_1(S),\gamma_1(S)]=\Phi(\gamma_1(S))$.
Thus $\gamma_1(S)$ is extraspecial.
\end{description}

Suppose any of the above conditions holds. Since $|S|=p^{p-1}$, Lemma \ref{exp.Si} implies that $S$ has exponent $p$.
Since $\gamma_1(S)\neq \C_S(\Z_2(S))$, by Lemma \ref{equal2step} we need $|S| \geq p^6$, and so $p\geq 7$.
\end{proof}

We close this section proving that if the subgroup $\gamma_1(S)$ is not abelian then the fusion system $\F$ contains only one type of pearls and the presence of distinct $\F$-classes of pearls gives information on the order of $S$.

\begin{theorem}\label{type.pearl}
Suppose that the group $\gamma_1(S)$ is not abelian.
Then
\[ \text{ either }\quad \mathcal{P}(\F) = \mathcal{P}(\F)_a \quad \text{ or } \quad \mathcal{P}(\F)= \mathcal{P}(\F)_e.\]
If moreover $\gamma_1(S)$ is not extraspecial and $E\in \mathcal{P}(\F)$ is a pearl then the following hold:
\begin{enumerate}
\item  $\Out_\F(S) = \N_{\Out_\F(S)}(E) \cong \C_{p-1}$ and $\Out_\F(E) \cong \SL_2(p)$;
\item if $|S|=p^n$ and there exists an $\F$-pearl $P$ such that $P\notin E^\F$  then $n \equiv \epsilon \mod (p-1)$, where $\epsilon = 0$ if $E$ is abelian and $\epsilon =1$ otherwise.
\end{enumerate}
\end{theorem}

\begin{proof}
Suppose the group $\gamma_1(S)$ is extraspecial.  Then by  Theorem \ref{S1extraspecial} we have $\mathcal{P}(\F)_a \neq \emptyset$ and $|S|=p^{p-1}$.
Aiming for a contradiction, suppose there exists an extraspecial pearl $E\in \mathcal{P}(\F)_e$ and let $\varphi_\lambda\in \Delta_\F(E)$. By Lemmas \ref{action.frat} and \ref{lambda.action}, the morphism $\varphi_\lambda$ centralizes $\Z(S)=\Phi(E)$ and acts as $\lambda$ on $\Z_2(S)/\Z(S)$ and as  $\lambda^{(p-1)-2}=\lambda^{-2}$ on $\gamma_1(S)/\gamma_2(S)$.
Take $z\in \Z_2(S)\backslash \Z(S)$ and $x \in \gamma_1(S) \backslash \gamma_2(S)$. Since $S$ has exponent $p$ we get
\[ [x,z]\varphi_\lambda = [x^{\lambda^{-2}},z^{\lambda}] =[x,y]^{\lambda^{-1}}.\]
Since $  \lambda^{-1} \not\equiv 1 \mod p-1$, we deduce that $[x,z]=1$ and  so $x\in \C_S(\Z_2(S))$. Thus $\gamma_1(S)=\langle x \rangle \gamma_2(S) =\C_S(\Z_2(S))$, contradicting the fact that $\Z(S)=\Z(\gamma_1(S))$.
Therefore $\mathcal{P}(\F)_e = \emptyset$ and  $\mathcal{P}(\F) = \mathcal{P}(\F)_a$.

Suppose $\gamma_1(S)$ is not extraspecial and $|S|=p^n$. Then by Theorem \ref{auto.S} the group $\Aut_\F(S)$ has order at most $p^{n-1}(p-1)$.
Let $E\in \mathcal{P}(\F)$ be a pearl. Since there exists some  $\varphi_\lambda \in \Delta_\F(E)$ and $\Out_\F(E)$ is isomorphic to a subgroup of $\GL_2(p)$ (Corollary \ref{auto.pearl}), the group $\N_{\Out_\F(E)}(\Out_S(E)) / \Out_S(E)$ is isomorphic to a subgroup of $\C_{p-1} \times \C_{p-1}$ containing $\langle \varphi_\lambda|_E \rangle \cong \C_{p-1}$. By Lemma \ref{lift} we conclude that $\Out_\F(E) \cong \SL_2(p)$ and \[\Out_\F(S) = \N_{\Out_\F(S)}(E) = \langle \varphi_\lambda\rangle \cong \C_{p-1}.\]

Suppose $P\in\mathcal{P}(\F)$ is a pearl. We want to show that $P$ is of the same type of $E$. Clearly this holds if $P\in E^\F$. Assume $P \notin E^\F$. Using what we proved above, we have $\Out_\F(S) = \N_{\Out_\F(S)}(P)$.
Let $M_E$ and $M_P$ be maximal subgroups of $S$ containing $E$ and $P$, respectively. By Lemma \ref{max.sbg} the groups $M_E$ and $M_P$ are uniquely determined and $M_E \neq M_P$.
Also by Theorem \ref{pearl} we have $M_E \neq \gamma_1(S) \neq M_P$. Note that $M_E$, $M_P$ and $\gamma_1(S)$ are maximal subgroups of $S$ normalized by $\Aut_\F(S)$. Therefore $\Aut_\F(S)$ acts as scalars on the quotient $S/\gamma_2(S) \cong \C_p \times \C_p$.
Let $\varphi_\lambda\in \Delta_\F(E)$ and $\varphi_\mu \in \Delta(P)$. Then $\varphi_\lambda$ acts on $\gamma_1(S)/\gamma_2(S)$ as $\lambda^{n-1-{\epsilon_E}}$ and on $M_E/\gamma_2(S)$ as $\lambda^{-1}$, where $\epsilon_E$ is equal to $0$ if $E$ is abelian and to $1$ otherwise. Since the action is scalar we need $n-1 -\epsilon_E \equiv -1 \mod (p-1)$ and so $n \equiv \epsilon_E \mod (p-1)$. Set $\epsilon_P=0$ if $P$ is abelian and $\epsilon_P=1$ otherwise. Since $\varphi_\mu$ acts as scalar on $S/\gamma_2(S)$, we also need $n \equiv \epsilon_P \mod (p-1)$. In particular $\epsilon_E = \epsilon_P$ and so $E$ and $P$ are either both abelian or both extraspecial. We proved that  $n \equiv \epsilon_E \mod (p-1)$ and either $\mathcal{P}(\F) = \mathcal{P}(\F)_a$ or $\mathcal{P}(\F)=\mathcal{P}(\F)_e$.
\end{proof}

\section{Proof of Theorem \ref{main} and simplicity of $\F$}
In this section we prove Theorem \ref{main} and we show that if $\F$ contains an abelian pearl and the subgroup $\gamma_1(S)$ of $S$ is neither abelian nor extraspecial then $O_p(\F)=1$. We also determine some sufficient conditions for $\F$ to be simple.  

\begin{lemma}\label{order.p-1}
Suppose that $p$ is an odd prime, $S$ is a $p$-group of sectional rank $k$ and $\F$ is a saturated fusion system on $S$ such that $\mathcal{P}(\F) \neq \emptyset$. If $p > k+1$ then $p^{k+1} \leq |S|\leq p^{p-1}$.

In particular if $p=k+2$ then $|S|=p^{k+1}=p^{p-1}$ and $S$ has a maximal subgroup that is elementary abelian of order $p^k$.
\end{lemma}

\begin{proof}
Aiming for a contradiction, suppose that $p > k+1$ and $|S|=p^n> p^{p-1}$.
Since $n\neq p-1$, by Theorem \ref{S1extraspecial} we get $\gamma_1(S)=\C_S(\Z_2(S))$. Thus Theorem \ref{positive.deg} implies $[\gamma_1(S),\gamma_j(S)2] \leq \gamma_{j+2}(S)$ for every $j\geq 2$.
Let $E\in \mathcal{P}(\F)$ be a pearl and let $\varphi_\lambda \in \Delta_\F(E)$.
Then $\varphi_\lambda$ acts on $\gamma_{n-(p-1)}(S)/\gamma_{n-(p-1)+1}(S)$ as $1$ if $E$ is abelian and as $\lambda^{-1}$ otherwise.
In particular if $x \in \gamma_{n-(p-1)}(S) \backslash \gamma_{n-(p-1)+1}(S)$ then $x$ commutes with every element of $\gamma_{n-(p-1)}(S)$. Thus $x \in \Z(\gamma_{n-(p-1)}(S))$ and since the members of the lower central series of $S$ are the only normal subgroups of $S$ contained in $\gamma_2(S)$, we deduce that $\Z(\gamma_{n-(p-1)}(S))=\gamma_{n-(p-1)}(S)$ and $\gamma_{n-(p-1)}(S)$ is abelian.
Using the action of $\varphi_\lambda$ we can also deduce that $\gamma_{n-(p-1)}(S)$ has exponent $p$ and so it is elementary abelian. However $|\gamma_{n-(p-1)}(S)| =p^{p-1} > p^k$, contradicting the assumptions.
Thus $|S|\leq p^{p-1}$.
Clearly if $S$ has sectional rank $k$ and contains a pearl then $|S|\geq p^{k+1}$. Therefore if $p>k+1$ then $p^{k+1} \leq |S|\leq p^{p-1}$.

Suppose $p=k+2$. Then $|S|=p^{p-1}=p^{k+1}$ and we conclude by Lemma \ref{lemma.1}.
\end{proof}

\begin{lemma}\label{pgeq2k}
Suppose that $p$ is an odd prime, $S$ is a $p$-group of sectional rank $k\geq 2$ and $\F$ is a saturated fusion system on $S$ such that $\mathcal{P}(\F) \neq \emptyset$. If $p \geq 2k+1$ (with equality only if $\mathcal{P}(\F)_e\neq \emptyset$) then $|S|=p^{k+1}$.
\end{lemma}

\begin{proof}
Suppose $p \geq 2k+1$.
By Lemma \ref{order.p-1} we have $|S|\leq p^{p-1}$. In particular by Lemma \ref{exp.Si}, the group $S$ has exponent $p$.
Aiming for a contradiction, assume that $|S|=p^n > p^{k+1}$. In particular we can consider the proper subgroup $\gamma_{n-(k+1)}(S)$ of $S$.
We prove that $\gamma_{n-(k+1)}(S)$ is abelian, and so it is an elementary abelian subgroup of $S$ of order $p^{k+1}$, contradicting the assumptions.

Let $E\in \mathcal{P}(\F)$ be a pearl and let $\varphi_\lambda \in \Delta_\F(E)$.
Set $\epsilon = 0$ if $E$ is abelian and $\epsilon =1$ otherwise.
Let $x\in \gamma_{n-(k+1)}(S)$ be such that $\gamma_{n-(k+1)}(S)=\langle x \rangle \gamma_{n-k}(S)$. Then by Lemma \ref{lambda.action} we have
 \[ x\varphi_\lambda = x^{\lambda^{k+1 -\epsilon}}.\]

Let $y\in \gamma_{n-k}(S)$ be such that $y\in \gamma_{n-i}(S) \backslash \gamma_{n-i+1}(S)$ for some $1 \leq i \leq k$. Then \[y\varphi_\lambda = y^{\lambda^{i-\epsilon}}.\]
If $[x,y] \neq 1$ then $[x,y] \in \gamma_{n-j}(S)\backslash \gamma_{n-(j-1)}(S)$ for some $j\geq 1$. In particular we get
\[ [x,y]^{\lambda^{j-\epsilon}} = [x, y] \varphi_\lambda = [x^{\lambda^{k+1-\epsilon}}, y^{\lambda^{i-\epsilon}}] \equiv [x,y]^{\lambda^{k + 1 + i- 2\epsilon}} \mod \gamma_{n-(j-1)}(S).\]

Thus we need $j  \equiv k + 1 + i- \epsilon \mod p-1$.

Note that $\gamma_{n-j}(S) \leq \gamma_{n-k}(S)$ so $j\leq k \leq p-2$. By assumption we also have
\[k+1 +i - \epsilon \leq 2k +1 -\epsilon \leq p-1\]
(indeed if $p=2k+1$ then  $\mathcal{P}(\F)_e\neq \emptyset$ and we can assume $\epsilon =1$). Hence we need $j = k +1 +i - \epsilon \geq k+1$, that is a contradiction.

Therefore $x$ commutes with $y$ and since $y$ was chosen arbitrarily we conclude that $x \in \Z(\gamma_{n-(k+1)}(S))$. Since $x\notin \gamma_{n-k}(S)$ we deduce that $\gamma_{n-(k+1)}(S)$ is abelian, contradicting the fact that $S$ has sectional rank $k$.
Therefore if $p \geq 2k+1$ (with equality only if $\mathcal{P}(\F)_e\neq \emptyset$) then $|S|=p^{k+1}$.
\end{proof}

\begin{proof}[\textbf{Proof of Theorem \ref{main}}] By Lemma \ref{self-centr.max.class} the group $S$ has maximal nilpotency class and by Corollary \ref{kleqp} we have $p\geq k$.

If $|S|=p^{k+1}$ then by Lemma \ref{lemma.1} the $p$-group $S$ has a maximal subgroup $M$ that is elementary abelian, and if $|S|\geq p^4$ then $M=\gamma_1(S)$.

Suppose $|S| \neq p^{k+1}$.  Then by
Corollary \ref{kleqp} and Lemmas \ref{order.p-1} and \ref{pgeq2k} either $p=k+1$ or  $k+3 \leq p \leq 2k+1$ (with $p=2k+1$ only if $\mathcal{P}(\F)_e = \emptyset$).
\begin{itemize}
\item If $p=k+1$ then $|S| \geq p^{k+2}=p^{p+1}$ and by Theorem \ref{S1extraspecial} we get $\gamma_1(S)=\C_S(\Z_2(S))$.
\item Suppose $k+3 \leq p \leq 2k+1$ (with $p=2k+1$ only if $\mathcal{P}(\F)_e= \emptyset$). Then by Lemma \ref{order.p-1} we get $p^{k+2} \leq |S| \leq p^{p-1}$ and by Lemma \ref{exp.Si} we conclude that $S$ has exponent $p$. In particular the group $\gamma_1(S)$ is not abelian since it has exponent $p$ and order at least $p^{k+1} > p^k$.
It remains to show that $k\geq 3$.
If $k=2$ then $p= 5$ and $|S|=5^4$. In particular the group $\C_S(\Z_2(S))$ is elementary abelian of order $5^3$, contradicting the fact that $S$ has sectional rank $2$. So we need $k\geq 3$.
\end{itemize}
\end{proof}

The next results study the simplicity of the fusion system $\F$, intended as in \cite[Definition I.6.1]{AKO}.

\begin{theorem}\label{simple}
Let $p$ be an odd prime, let $S$ be a $p$-group of order $|S|=p^n\geq p^4$ and let $\F$ be a saturated fusion system on $S$ such that $\mathcal{P}(\F) \neq \emptyset$.
If $\F$ contains an abelian pearl then $O_p(\F)=1$.
 Suppose moreover that the subgroup $\gamma_1(S)$ of $S$  is neither abelian nor extraspecial. Then
\begin{enumerate}
\item  if there exists a unique $\F$-conjugacy class of $\F$-essential subgroups and $n\not\equiv 1 \mod (p-1)$ then $\F$ is simple;
\item  if $S$ does not have sectional rank $p-1$ and $\F=\F_S(G)$ for some finite non-abelian  simple group $G$,  then $G$ is not alternating nor sporadic. 
 \end{enumerate}
\end{theorem}

\begin{proof} Let $E \in \mathcal{P}(\F)_a$ be an abelian pearl.
Note that $O_p(\F)$ is a normal subgroup of $S$ that is contained in every $\F$-essential subgroup of $S$ (\cite[Proposition I.4.5]{AKO}). In particular $O_p(\F) \leq E$. Since $|S|\geq p^4$, the group $E$ is not normal in $S$ and so either $O_p(\F)=\Z(S)$ or $O_p(\F)=1$. 
Since $\Z(S)$ is not normalized by $\Aut_\F(E)$, we deduce that $O_p(\F)=1$.
\begin{enumerate}
\item  Suppose that $E^\F$ is the only $\F$-conjugacy class of $\F$-essential subgroups. Let $\E$ be a non-trivial normal subsystem of $\F$ (as defined in \cite[Definition I.6.1]{AKO}). We want to prove that $\E=\F$.
Suppose that $\E$ is defined on the subgroup $P$ of $S$. Then $P$ is strongly $\F$-closed. In particular $P\norm S$ and $\Z(S) \leq P$. Since $E=\langle \Z(S)^{\Aut_\F(E)}\rangle$, we deduce that $E \leq P$ and so either $P=S$ or $P$ is the unique maximal subgroup of $S$ containing $E$. Also, by the Frattini condition in the definition of normal subsystem (\cite[Definition I.6.1]{AKO}) we deduce that $\Out_\E(E) = \Out_\F(E)\cong \SL_2(p)$ (Theorem \ref{type.pearl}(1)). In particular $E$ and every $\F$-conjugate of $E$ contained in $P$ are $\E$-essential subgroups of $P$.  Suppose $P < S$ and let $g\in S \backslash P$.  The groups $E$ and $E^g$ are $\E$-essential subgroups of $P$ and $E^g\notin E^\E$. By Lemma \ref{max.sbg} there exist unique maximal subgroups $M_1$ and $M_2$ of $P$ such that $E\leq M_1$, $E^g\leq M_2$ and $M_1\neq M_2$. Note that $\gamma_3(S)=[P,P]\leq M_1 \cap M_2$ and $M_1 \neq \gamma_2(S) \neq M_2$. Let $\varphi_\lambda\in \Delta_\E(E)\neq \emptyset$. We have  $\Out_\E(P) \leq \Out_\F(P)= \Out_S(P)\N_{\Out_\F(P)}(E)$ (Remark \ref{auto.remark}) and $\N_{\Out_\F(P)}(E) \cong \C_{p-1}$ (Theorem \ref{type.pearl}(1)). Therefore   $\Out_\E(P) = \N_{\Out_\E(P)}(E)\cong \C_{p-1}$ and the automorphism $\varphi_\lambda$ acts as scalars on $P/\gamma_3(S)$. We can assume that $\varphi_\lambda\in \Delta_\F(E)$ and so it acts on $S$ as described by Lemma \ref{lambda.action}. In particular $\varphi_\lambda$ acts as $\lambda^{-1}$ on $\gamma_2(S)/\gamma_3(S)$ if and only if $n-2 \equiv -1 \mod (p-1)$, that is $n \equiv 1 \mod (p-1)$ and contradicts the assumptions.
Therefore we get $P=S$. Note that $E^\E = E^\F$, $\Out_\E(E)\cong \Out_\F(E)\cong \SL_2(p)$ and $\Out_\E(S) \cong \Out_\F(S)$  (Theorem \ref{type.pearl}(1)). Since $E^\F$ is the only $\F$-conjugacy class of $\F$-essential subgroups of $S$, the Alperin-Goldschmidt fusion theorem (\cite[Theroem I.3.5]{AKO}) guarantees that $\F$ is completely determined by $\Aut_\F(S)$ and $\Aut_\F(E)$. Therefore $\E = \F$ and the fusion system $\F$ is simple.

\item
By assumption the group $\gamma_1(S)$ is neither abelian nor extraspecial and $S$ does not have sectional rank $p-1$. So by Theorem \ref{main} we have $p^5 \leq |S|<p^{p-1}$ and $S$ has exponent $p$. Since $S$ is not abelian, the group $G$ is not abelian and if it is alternating then $G=\Alt(m)$ with $m\ge p^2$, that implies $|S|\geq p^{p+1}$ and gives a contradiction.
Also, there is no sporadic group with a Sylow $p$-subgroup of order $p^5 \leq |S| < p^{p-1}$ for any $p$.
\end{enumerate}
\end{proof}

\begin{cor}\label{simple.pearl}
Let $p$ be an odd prime, and let $\F$ be a saturated fusion system on $S$ such that $\mathcal{P}(\F)_a \neq \emptyset$.
Suppose $p$ is as in part $(3)$ of Theorem \ref{main} and $\gamma_1(S)$ is not extraspecial.
If all the $\F$-essential subgroups of $S$ are pearls then $\F$ is simple.
\end{cor}

\begin{proof}
By assumption $S$ has exponent $p$ and $p^{k+2} \leq |S|\leq p^{p-1}$, where $k$ is the sectional rank of $S$. In particular the group $\gamma_1(S)$ is not abelian. Since the group $\gamma_1(S)$ is not extraspecial, by Theorem \ref{S1extraspecial} we get $|S|<p^{p-1}$. Therefore Theorem \ref{type.pearl}(2) implies that there is a unique $\F$-conjugacy class of pearls. Since all the $\F$-essentials subgroups of $S$ are pearls,  we conclude by Theorem \ref{simple}(1) that $\F$ is simple.
\end{proof}

\section{Essential subgroups of $p$-groups of maximal nilpotency class}

Let $p$ be an odd prime and let $S$ be a $p$-group having maximal nilpotency class and order $|S|=p^n\geq p^4$.
Let $E$ be an $\F$-essential subgroup of $S$. Then by Lemma \ref{pearl} either $E$ is a pearl or $E$ is contained in $\gamma_1(S)$ or $\C_S(\Z_2(S))$.
In this section we focus on $\F$-essential subgroups of $S$ that are not pearls.

First notice that if $\gamma_1(S)$ is abelian then none of its proper subgroups can be $\F$-essential, since $\F$-essential subgroups are $\F$-centric.

\begin{lemma}\label{no.ess.inside.extra}
Suppose that the group $\gamma_1(S)$ is extraspecial. Then no proper subgroup of $\gamma_1(S)$ is $\F$-essential.
\end{lemma}

\begin{proof}
Aiming for a contradiction, suppose there exists a subgroup $E < \gamma_1(S)$ that is $\F$-essential.
Note that $\Phi(\gamma_1(S)) =\Z(\gamma_1(S)) = \Z(S) \leq E$, since $E$ is $\F$-centric. Thus $E\norm \gamma_1(S)$ and by Lemma \ref{strict.frattini} we get that $E$ is elementary abelian. Since $\gamma_1(S)$ is extraspecial of order $p^{n-1}$, this implies   $|E|\leq p^{n/2}$. In particular the quotient $\gamma_1(S)/E$ is elementary abelian of order $[\gamma_1(S) \colon E] \geq p^{(n-2)/2}$. On the other hand, \cite[Theorem 6.9]{Sambale} implies that $[\gamma_1(S) \colon E]\leq p^{n/4}$. Thus $n= 4$ and  $\gamma_1(S)=\C_S(\Z_2(S))$ is abelian by Theorem \ref{equal2step}, contradicting the fact that $\gamma_1(S)$ is extraspecial.
\end{proof}

\begin{lemma} \label{Si.in.ess}
Suppose that $|S|> p^{p+1}$ and let $l$ be the degree of commutativity of $S$.
Let $E\leq S$ be an $\F$-essential subgroup and suppose that $E\leq \gamma_1(S)$. Fix $2\leq  i \leq n-1$.
If $\gamma_i(S)\nleq E$ then  $l \leq (p-2)-i$ and $|S|\leq p^r$ where $r=4(p -2) - 2i$.
\end{lemma}

\begin{proof}
Consider the following sequence of characteristic subgroups of $E$:
\begin{equation}\label{om.series} 1 < \Omega_1(E) \leq \dots \leq \Omega_m(E) =E.\end{equation}
By Corollary \ref{omegas}, for every $j\leq m$ we have $\Omega_j(E)= E \cap \Omega_j(\gamma_1(S))$.
Suppose $\gamma_i(S)\nleq E$. In particular $E < N=\N_{\gamma_i(S)E}(E)$ and so $\Out_N(E)\cong N/E \neq 1$. Assume by contradiction that $l \geq (p-1)-i$. Then by Lemma  \ref{omega.series} the group $\N_{\gamma_i(S)}(E)$ stabilizes the series (\ref{om.series}). Hence by Lemma \ref{char.series} we get $\Aut_{N}(E) = \Inn(E)$, that is a contradiction. Thus we need $l\leq  (p-2)-i$.
The bound on the order of $S$ follows from  Theorem \ref{bound.deg.comm}(1).
\end{proof}

The previous lemma will be useful to prove Theorem \ref{small.rank}.

\begin{lemma}\label{ess.in.C2}
Suppose that $|S|\geq p^4$ and $\gamma_1(S)\neq \C_S(\Z_2(S))$. Let $E\leq S$ be an $\F$-essential subgroup contained in $\C_S(\Z_2(S))$ but not in $\gamma_1(S)$. Then $p^3 \leq |E| \leq p^5$, $E$ has exponent $p$ and one of the following holds:
\begin{enumerate}
\item $E\cong \C_p \times \C_p \times \C_p$; or
\item $E \cong \C_p \times p^{1+2}_+$ and $\Z(E)=\Z_2(S)$; or
\item $E/\Z_2(S) \cong p^{1+2}_+$ and $\Z(S)$ is not normalized by $\Aut_\F(E)$.
\end{enumerate}
In particular if $E=\C_S(\Z_2(S))$ then $E/\Z_2(S)\cong p^{1+2}_+$, $\Z_3(S)=\Phi(E)$ and $|S|=p^6$.
\end{lemma}

\begin{proof}
Note that $\Z_2(S) < E$ because $E$ is $\F$-centric and $\Z_2(S)$ is not $\F$-essential, so $|E|\geq p^3$.
Suppose $|E|=p^3$. Then $E$ is abelian. If it is not elementary abelian then the series $1 < E^p < \Omega_1(E) < E$ is stabilized by $\N_S(E)$, contradicting Lemma \ref{char.series}. Thus $E$ is elementary abelian, $E \cong \C_p \times \C_p \times \C_p$. Also note that $E < \C_S(\Z_2(S))$ otherwise $\C_S(\Z_2(S))=\gamma_1(S)$, contradicting the assumptions.

Suppose $|E| \geq p^4$. Then by Lemma \ref{sbg.max.class} the quotient $E/\Z(S)$ has maximal nilpotency class and $\Z_3(S) \leq E$. In particular, since $E\nleq \gamma_1(S)$, we have $[E,\Z_3(S)]\Z(S) = \Z_2(S)$ and $\Z_2(S)=\Z(E)$ (so $E$ is not abelian).

If $\Z(S)$ is normalized by $\Aut_\F(E)$ then by Lemma \ref{max.class.quotient} we conclude $E/\Z(S) \cong p^{1+2}_+$. In particular $|E|=p^4$ and by Theorem \ref{equal2step} we deduce that $E < \C_S(\Z_2(S))$.
Suppose   $\Z(S)$ is not normalized by $\Aut_\F(E)$. Note that the quotient $E/\Z_2(S)$ has maximal nilpotency class and $\Z_2(S)$ is normalized by $\Aut_\F(E)$.

If $E < \C_S(\Z_2(S))$ then $E < N=\N_{\C_S(\Z_2(S))}(E)$ and $N$ centralizes $\Z_2(S)$.
Therefore by Lemma \ref{max.class.quotient}  we conclude that $E/\Z_2(S)$ is isomorphic to either $\C_p \times \C_p$ or $p^{1+2}_+$.

Suppose that $E=\C_S(\Z_2(S))$. Since the members of the lower central series of $S$ are the only normal subgroups of $S$ of index greater than $p$ in $S$, we deduce that $[E, \Z_3(S)]=\Z_2(S)$. In particular $\Z_2(S) \leq \Phi(E)$. Hence by Lemma \ref{max.class.quotient} we get $E/\Z_2(S) \cong p^{1+2}_+$. In particular $|E|=p^5$, $\Z_3(S)=\Phi(E)$ and $|S|=p^6$.

We now show that if $p^4 \leq |E| \leq p^5$ then $E$ has exponent $p$. Note that the assumption $\gamma_1(S)\neq \C_S(\Z_2(S))$ implies $p\geq 5$ (by Theorem \ref{equal2step}). In particular $E$ is a regular group and so every element of $\Omega_1(E)$ has order $p$. Suppose for a contradiction that $E$ does not have exponent $p$. Then $\Omega_1(E) < E$.
Thus by Lemma \ref{power}, either $|E|=p^4$ and $\Omega_1(E)=\Z_3(S)$ or $|E|=p^5$ and $\Omega_1(E)=\Z_4(S)$.
If $|E|=p^4$ then $|[E,E]| =p$ (because $[E \colon \Z(E)]=p^2$) and so the series $1 < [E,E] < \Z(E) < \Omega_1(E) < E$ is stabilized by $\N_S(E)$, contradicting Lemma \ref{char.series}. If $|E|=p^5$ then $E/\Z_2(S)\cong p^{1+2}_+$ and so $\Z_3(S)=\Z_2(S)\Phi(E)=\Z(E)\Phi(E)$. Thus the series $\Z_2(S) \cap \Phi(E) \leq \Z_2(S) < \Z_2(S)\Phi(E) < \Omega_1(E)<E$ is stabilized by $\N_S(E)$, and again we get a contradiction by Lemma \ref{char.series}.
Thus $E$ has exponent $p$. In particular if $|E|=p^4$ then $E \cong \C_p \times p^{1+2}_+$.
\end{proof}

We conclude characterizing the $\F$-essential subgroups of $S$ when the group $\gamma_1(S)$ is extraspecial.

\begin{theorem}\label{ess.extra}
Suppose that $p$ is an odd prime, $\F$ is a saturated fusion system on a $p$-group $S$ and $\mathcal{P}(\F)\neq \emptyset$ (so $S$ has maximal nilpotency class).
Suppose that $\gamma_1(S)$ is extraspecial and let $\E$ be the set of $\F$-essential subgroups of $S$. Then $p\geq 7$, $S$ has order $p^{p-1}$ and exponent $p$ and
\[ \E \subseteq \{ \gamma_1(S) \} \cup \mathcal{P}(\F)_a \cup \{E\leq \C_S(\Z_2(S)) \mid E\nleq \gamma_1(S) \text{ and } p^3 \leq |E| \leq p^5\}. \]

	If moreover $p=7$  then $S$ is isomorphic to a Sylow $7$-subgroup of the group $\G_2(7)$ (and this is always the case when $\C_S(\Z_2(S))$ is $\F$-essential).
\end{theorem}

\begin{proof}
By assumption $\mathcal{P}(\F)\neq \emptyset$ and $\gamma_1(S)$ is extraspecial. Hence by Theorems \ref{S1extraspecial} and \ref{type.pearl} we have $\mathcal{P}(\F) =\mathcal{P}(\F)_a$, $p\geq 7$, and $S$ has order $p^{p-1}$ and exponent $p$.

Let $E\leq S$ be an $\F$-essential subgroup. By Lemmas \ref{pearl}, \ref{no.ess.inside.extra} and \ref{ess.in.C2} one of the following holds:
\begin{itemize}
\item $E \in \mathcal{P}(\F)_a$;
\item $E=\gamma_1(S)$;
\item $E\leq \C_S(\Z_2(S))$, $E\nleq \gamma_1(S)$, $p^3 \leq |E|\leq p^5$ and if $E=\C_S(\Z_2(S))$ then $|S|=p^6$.
\end{itemize}

Note that if   $E=\C_S(\Z_2(S))$ then the fact that $|S|=p^{p-1}$ implies $p=7$.

Finally, one can show using the computer software \emph{Magma} that there exists a unique  (up to isomorphism) group $P$ having order $7^6$, nilpotency class $5$, exponent $7$ and a maximal subgroup that is extraspecial ($P$ is isomorphic to the group listed in the \texttt{SmallGroups} library as \texttt{SmallGroup(7\string^6, 807)}). Also such $P$ is isomorphic to a  Sylow $7$-subgroup of the group $\G_2(7)$. Thus if $p=7$ then $S$  is isomorphic to a  Sylow $7$-subgroup of the group $\G_2(7)$.
\end{proof}

Fusion systems on Sylow $p$-subgroups of the group $\G_2(p)$ have been classified in \cite{G2p}.
More generally, fusion systems on $p$-groups having a maximal subgroup that is extraspecial are the subject of study of Moragues-Moncho (\cite{Raul}).

\section{Fusion systems on $p$-groups of small sectional rank containing pearls}
Let $p$ be an odd prime, let $S$ be a $p$-group having sectional rank $k \leq 4$ and let $\F$ be a saturated fusion system on $S$ such that $\mathcal{P}(\F) \neq \emptyset$. In particular $S$ has maximal nilpotency class (Lemma \ref{self-centr.max.class}).

\begin{lemma}\label{small.rank.1}
One of the following holds:
\begin{enumerate}
\item $|S|=p^{k+1}$;
\item $(k,p)=(2,3)$ or $(k,p)=(4,5)$, $|S|\geq p^{p+1}$ and $\gamma_1(S) = \C_S(\Z_2(S))$;
\item $3\leq k\leq 4$, $p=7$, $S$ has exponent $7$ and
\begin{enumerate}
\item if $k=3$ then $|S|=7^5$;
\item if $k=4$ then $|S|=7^6$ and if $\mathcal{P}(\F)_a \neq \emptyset$ then $S$ is isomorphic to a Sylow $7$-subgroup of the group $\G_2(7)$.
\end{enumerate}
 \end{enumerate}
\end{lemma}

\begin{proof}
By Theorem \ref{main}, we only need to prove statements $3(a)$ and $3(b)$.
\begin{enumerate}[$(a)$]
\item Suppose $k=3$. Then by Theorem \ref{main} we have $7^5 \leq |S| \leq 7^6$, with $|S|=7^6$ only if $\mathcal{P}(\F)_e=\emptyset$. Aiming for a contradiction, assume $|S|=7^6$. Then by Theorem \ref{S1extraspecial} the group $\gamma_1(S)$ is extraspecial (note that it is not abelian because $S$ has exponent $7$ and sectional rank $3$).  In particular the quotient $\gamma_1(S)/\Phi(\gamma_1(S)) = \gamma_1(S)/\Z(S)$ is an elementary abelian section of $S$ having order $7^4$, contradicting the fact that $k=3$. Thus we need $|S|=7^5$.

\item Suppose $k=4$. Then by Theorem \ref{main} we have $|S|=7^6$.  If $\mathcal{P}(\F)_a \neq \emptyset$ then the group $\gamma_1(S)$ is extraspecial by Theorem \ref{S1extraspecial} and by Theorem \ref{ess.extra} the group $S$ is isomorphic to a Sylow $7$-subgroup of the group $\G_2(7)$.
\end{enumerate}
\end{proof}

If $k=2$ and $|S|=p^3$ then $S\cong p^{1+2}_+$ by \cite[Theorem 4.2]{Stancu} and the saturated fusion systems on $S$ have been classified in \cite{RV}.
If $(k,p)=(2,3)$ then the saturated fusion systems on the $3$-group $S$ have been classified in \cite{DRV}. We now focus on $k\geq 3$.

\begin{lemma}\label{Sp4p}
If $k=3$ and $|S|=p^{k+1}=p^4$  then $S$ is isomorphic to a Sylow $p$-subgroup of the group $\Sp_4(p)$.
\end{lemma}

\begin{proof}
By Theorem \ref{main} the group $\gamma_1(S)$ is elementary abelian, $\gamma_1(S) \cong \C_p \times \C_p \times \C_p$.
 By assumption there exists a pearl $E\in \mathcal{P}(\F)$. Let $x\in E \backslash \gamma_1(S)$. Then $x$ has order $p$, $[\gamma_1(S), x] = \gamma_2(S)$ and $[\gamma_2(S),x] = \gamma_3(S) = \Z(S)$. Let $1 \neq z\in \Z(S)$, $u\in \gamma_2(S) \backslash \Z(S)$ and $v\in \gamma_1(S) \backslash \gamma_2(S)$.  Then $\mathcal{B} = \{ z, u, v \}$ is a basis for $\gamma_1(S)$ and with respect to $\mathcal{B}$ the element $x$ acts on $\gamma_1(S)$ as the matrix
\[\begin{pmatrix} 1 & 1 & 0 \\ 0 & 1 & 1 \\ 0 & 0 & 1 \end{pmatrix}.\]
Thus $S \cong \gamma_1(S) \colon \langle x \rangle$ is isomorphic to a Sylow $p$-subgroup of the group $\Sp_4(p)$.
\end{proof}

\begin{lemma}\label{sec.3.exotic}
Suppose $k=3$, $p=7$ and $S$ has order $7^5$. Then $S$ is isomorphic to the group listed in \emph{Magma} as \rm{\texttt{SmallGroup(7\string^5, 37)}}, $\mathcal{P}(\F)= \mathcal{P}(\F)_a$, $\F$ is simple and there exists an abelian pearl $E\in \mathcal{P}(\F)_a$ such that
\begin{itemize}
\item $E^\F$ is the unique $\F$-conjugacy class of $\F$-essential subgroups of $S$;
\item $\Aut_\F(E) \cong \SL_2(7)$ and $\Out_\F(S) = \N_{\Out_\F(S)}(E) \cong \C_6$.
\end{itemize}
If moreover we assume the classification of finite simple groups then $\F$ is exotic.
\end{lemma}

\begin{proof}
By Theorem \ref{main} we get that $\mathcal{P}(\F)=\mathcal{P}(\F)_a$, $S$ has exponent $7$ and $\gamma_1(S)$ is not abelian.
Note that by Theorem \ref{equal2step} and the fact that $|S|=7^5$ we conclude $\gamma_1(S)=\C_S(\Z_2(S))$.
Using the computer software \emph{Magma}, one can show that the group  $P=$ \texttt{SmallGroup(7\string^5, 37)} is (up to isomorphism) the unique group of order $7^5$ having nilpotency class $4$, exponent $7$ and such that the group $\C_S(\Z_2(S))$ is not abelian. Therefore $S\cong P$.

Since $|S|=7^5$, by Theorem \ref{S1extraspecial} the group $\gamma_1(S)$ is not extraspecial and so by Theorem \ref{type.pearl}(2) there exists a unique $\F$-conjugacy class of pearls, say $E^\F$ for $E\in \mathcal{P}(\F)_a$.  Also Theorem \ref{type.pearl}(1) implies that $\Aut_\F(E)\cong \SL_2(7)$ and $\Out_\F(S) = \N_{\Out_\F(S)}(E) \cong \C_6$.

Aiming for a contradiction, suppose that the group $\gamma_1(S)$ is $\F$-essential.
Since $\Phi(\gamma_1(S))=\Z(S) < \Z_2(S)=\Z(\gamma_1(S))$, by Theorem \ref{auto.rank.3} we deduce that $O^{7'}(\Out_\F(\gamma_1(S)))\cong \SL_2(7)$ and $\Out_\F(\gamma_1(S)) \leq \GL_2(7) \times \GL_1(7)$. Since $\Out_\F(S) \cong \C_6$ and every morphism in $\N_{\Aut_\F(\gamma_1(S))}(\Aut_S(\gamma_1(S)))$ is the restriction of a morphism of $\Aut_\F(S)$, we deduce that $\Out_\F(\gamma_1(S))\cong \SL_2(p)$. Take $\varphi_\lambda\in \Delta_\F(E)$. Then $\varphi_\lambda|_{\gamma_1(S)} \in \Aut_\F(\gamma_1(S))$ and  by Lemma \ref{lambda.action} the morphism $\varphi_\lambda$ acts as $\begin{pmatrix} \lambda^4 & 0 \\ 0 & \lambda^3 \end{pmatrix}$ on $\gamma_1(S)/\Z_2(S)$. Such matrix has determinant $\lambda^7 \not\equiv 1 \mod p$, contradicting the fact that $\Out_\F(\gamma_1(S))\cong \SL_2(p)$. Thus the group $\gamma_1(S)$ is not $\F$-essential.

Suppose for a contradiction that there exists an $\F$-essential subgroup $P\leq S$ that is not an $\F$-pearl. Hence by Lemma \ref{pearl} and the fact that $\gamma_1(S)$ is not $\F$-essential   we get $P < \gamma_1(S)$. Thus $\Z(\gamma_1(S))=\Z_2(S) < P < \gamma_1(S)$ and so $P\cong \C_7 \times \C_7 \times \C_7$. By Theorem \ref{auto.rank.3} we also have $[\N_S(P) \colon P]=7$, so $\gamma_1(S)=\N_S(P)$ and $P\neq \gamma_2(S)$.
Take $\varphi_\lambda\in \Delta_\F(E)$. Then $\varphi_\gamma$ acts on $\gamma_1(S)/\Z_2(S)\cong \C_p \times \C_p$ and normalizes $\gamma_2(S)/\Z_2(S)$. By Maschke's Theorem (\cite[Theorem 3.3.2]{Gor}) there exists a maximal subgroup $T/\Z_2(S)$ of $\gamma_1(S)/\Z_2(S)$ that is distinct from $\gamma_2(S)/\Z_2(S)$ and normalized by $\varphi_\gamma$. Note that $T=P^g$ for some $g\in S$ and $T$ is $\F$-essential by Theorem \ref{auto.rank.3}. Upon replacing $P$ with $T$, we may assume that $P$ is normalized by $\varphi_\lambda$. In particular $\varphi_\gamma$ acts on $P/\Z_2(S)$ as it does on $\gamma_1(S)/\gamma_2(S)$.  Note that every morphism in $\N_{\Aut_\F(P)}(\Aut_S(P))$ is the restriction of a morphism of $\Aut_\F(S)$ and by Theorem \ref{auto.rank.3} we get that either $\Out_\F(P) \cong \SL_2(7)$ or $\Out_\F(P)\cong \PSL_2(7)$. In particular $\varphi_\lambda$ acts on $P$ as a matrix having determinant $1$ (modulo $7$). However this contradicts Lemma \ref{lambda.action}.

Thus all the $\F$-essential subgroups of $S$ are pearls.
Hence by Corollary \ref{simple.pearl} we conclude that $\F$ is simple.

Assume the classification of finite simple groups and suppose by contradiction that $\F$ is not exotic. Then by Theorem \ref{simple}(2) there exists a finite simple group $G$ of Lie type that realizes $\F$. In particular $S$ is isomorphic to a Sylow $7$-subgroup of $G$. Since there is no simple group of Lie type in defining characteristic $7$ having a Sylow $7$-subgroup of order $7^5$ (see \cite[Theorem 2.2.9 and Table 2.2]{GLS3}), we deduce that $G$ is of cross-characteristic $7$. Note that the group $\gamma_2(S)$ is elementary abelian of order $7^3$.  By \cite[Theorem 4.10.3]{GLS3} the group $\gamma_2(S)$ has to be the unique elementary abelian subgroup of order $7^3$. However every maximal subgroup of $\gamma_1(S)$ containing $\gamma_3(S)$ is elementary abelian of order $7^3$, giving a contradiction. Therefore the fusion system $\F$ is exotic.
\end{proof}

\begin{remark} It can be checked with \emph{Magma} that the group $P=$ \texttt{SmallGroup(7\string^5, 37)} is isomorphic to a maximal subgroup of a Sylow $7$-subgroup of the group $\G_2(7)$.
\emph{Magma} also enables us to show that there exists an automorphism $\varphi$ of $P$ that normalizes a self-centralizing subgroup $E$ of $P$ isomorphic to $\C_7 \times \C_7$ and acts on it as the matrix $\begin{pmatrix} 3 & 0 \\ 0 & 5 \end{pmatrix}$. By \cite[Proposition 5.1]{BLO} the fusion system $\E$ determined by $\Inn(P)\langle \varphi\rangle$ and the subgroup of $\Aut(E)$ isomorphic to $\SL_2(7)$ (containing $\varphi$)  is saturated (and $E$ is an $\E$-essential subgroup). This implies that the fusion system described in Lemma \ref{sec.3.exotic} exists.
\end{remark}

\begin{lemma}\label{7.6}
Suppose $k=4$, $p=7$, $S$ has order $7^6$ and $\mathcal{P}(\F)=\mathcal{P}(\F)_e$. Then $S$ is isomorphic to the group listed in \emph{Magma} as \rm{\texttt{SmallGroup(7\string^6, 813)}} and there exists a pearl $E\in \mathcal{P}(\F)_e$ such that $E^\F$ is the unique $\F$-conjugacy class of $\F$-essential subgroups of $S$. In particular $O_7(\F) = \Z(S)$.
\end{lemma}

\begin{proof}
Since $\mathcal{P}(\F)=\mathcal{P}(\F)_e$, Theorem \ref{S1extraspecial} implies that $\gamma_1(S)$ is not extraspecial and $\gamma_1(S) = \C_S(\Z_2(S))$. In particular by Theorem \ref{positive.deg} we have $[\gamma_2(S),\gamma_2(S)] = [\gamma_2(S),\gamma_3(S)] \leq \gamma_6(S) =1$. So the group $\gamma_2(S)$ is abelian.
By Theorem \ref{main} the group $S$ has exponent $7$. Hence $\gamma_2(S)$ is elementary abelian.
Theorem \ref{main} also tells us that the group $\gamma_1(S)$ is not abelian.

Using the computer software \emph{Magma} we can prove that the groups \texttt{SmallGroup(7\string^6, 789)} and \texttt{SmallGroup(7\string^6, 813)} are, up to isomorphism, the only groups of order $7^6$ that have nilpotency class $5$, exponent $7$, elementary abelian derived subgroup and such that the centralizer of the second center is not abelian.
Thus either $S\cong $ \texttt{SmallGroup(7\string^6, 789)} or $S\cong$\texttt{SmallGroup(7\string^6, 813)}.

Using \emph{Magma} we can determine the generators of the automorphism group of the $7$-group $P=$ \texttt{SmallGroup(7\string^6, 789)}. We find $6$ generators having order $7$ and one having order $6$. Also, all the generators of order $7$ act trivially on $\Z(P)$ and the one of order $6$ acts on $\Z(P)$ as $5$. In particular there is no automorphism of $P$ of order prime to $7$ that acts trivially on $\Z(P)$.
Let $E\in \mathcal{P}(\F)_e$ be an extraspecial pearl. Then by Lemma \ref{action.frat} every morphism in $\Delta_\F(E)$ centralizes $\Z(S)$. If $S \cong P$ then $\Delta_\F(E) = \emptyset$, that is a contradiction.

Thus we need $S\cong$\texttt{SmallGroup(7\string^6, 813)}.
Note that  we can check that the automorphism group of \texttt{SmallGroup(7\string^6, 813)} contains an automorphism of order $6$ centralizing the center.

We now show that $E^\F$ is the only $\F$-conjugacy class of $\F$-essential subgroups.
By Theorem \ref{type.pearl}(2), $E^\F$ is the only $\F$-conjugacy class of pearls.

Suppose the group $\gamma_1(S)$ is $\F$-essential.
Note that $\gamma_1(S)$ has rank $3$ and Theorem \ref{auto.rank.3} and the fact that $\Out_\F(S) \cong \C_6$ (Theorem \ref{type.pearl}(1)) imply that either $\Out_\F(\gamma_1(S))\cong \SL_2(7)$ or $\Out_\F(\gamma_1(S))\cong \PSL_2(7)$.
Take $\varphi_\lambda\in \Delta_\F(E)$. Then $\varphi_\lambda$ acts on $V/\Phi(\gamma_1(S))$ as a matrix having determinant $1$, where $V\leq \gamma_1(S)$ is a subgroup of $\gamma_1(S)$ normalized by $\Out_\F(\gamma_1(S))$ (and so normal in $S$). This contradicts Lemma \ref{lambda.action}. Hence the group $\gamma_1(S)$ is not $\F$-essential.

Aiming for a  contradiction, suppose that $P\leq S$ is an $\F$-essential subgroup that is not a pearl. Then by Lemma \ref{pearl} and what we proved above we get $P<\gamma_1(S)=\C_S(\Z_2(S))$. In particular $\Z_2(S) < P$ and so $7^3\leq |P| \leq 7^4$.
Note that $\Phi(\gamma_1(S))=\Z_2(S)=\Z(\gamma_1(S)) \leq P$, so $P\norm \gamma_1(S)$.
By Theorems \ref{auto.rank.3} and \ref{samb4} either $[\N_S(P) \colon P]=7$ or $\Out_\F(P)$ involves $\PSL_2(49)$, $P$ has rank $4$ and $[\N_S(P) \colon P]=7^2$. Since $|\gamma_1(S)|=7^5$ and $\gamma_1(S)\leq \N_S(P)$ we get  $[\gamma_1(S) \colon P] \leq 7$.
If $P\norm S$ then every morphism in $\N_{\Aut_\F(P)}(\Aut_S(P))$ is the restriction of a morphism of $\Aut_\F(S)$ and so we need $[S \colon P]=7$ and $P=\gamma_1(S)$, that is a contradiction. Thus $|P|=7^4$ and  $\gamma_1(S)=\N_S(P)$ (that implies $P\neq \gamma_2(S)$).
Note that $\gamma_1(S)=P\gamma_2(S)$ and $\gamma_2(S)$ is abelian. If $P$ is abelian then $[\gamma_1(S) \colon \Z(\gamma_1(S))]=[\gamma_1(S) \colon \gamma_2(S) \cap P] =7^2$, contradicting the fact that $\Z(\gamma_1(S))=\Z_2(S)$. So $P$ is not abelian. Thus $\Phi(P)\neq 1$ and $\Phi(\gamma_1(S))=\Z_2(S)=\Z(P)$. Hence the group $\gamma_1(S)$ stabilizes the series $\Phi(P) < \Z(P) <P$, contradicting Lemma \ref{char.series}.

Hence all the $\F$-essential subgroups of $S$ are pearls and by what we showed above this implies that $E^\F$ is the only $\F$-conjugacy class of $\F$-essential subgroups. In particular $O_7(\F)=\Z(S)=\Z(E)$.
\end{proof}

\begin{lemma}\label{4.5}
Suppose $k=4$, $p=5$ and $|S|=p^n\geq 5^6$. Let $P\leq S$ be an $\F$-essential subgroup of $S$. Then either $P$ is a pearl or $P=\gamma_1(S)=\C_S(\Z_2(S))$.
\end{lemma}

\begin{proof}
Since $|S|\geq 5^6$, by Lemma \ref{equal2step} we have $\gamma_1(S)=\C_S(\Z_2(S))$. Suppose $P$ is not a pearl.
Then by Lemma \ref{pearl} we have $P \leq \gamma_1(S)$.

Aiming for a contradiction suppose that $P < \gamma_1(S)$. In particular $\gamma_1(S)$ is not abelian.
By Theorem \ref{S1extraspecial}, the group $\gamma_1(S)$ is not extraspecial. Since $\mathcal{P}(\F)\neq \emptyset$, by Theorem \ref{type.pearl}(1) we have $|\Out_\F(S)|=4$.

By Theorem \ref{samb4} we have $[\N_S(P) \colon P] \leq 5^2$ and if $[\N_S(P) \colon P] = 5^2$ then the group $\Out_\F(P)$ involves $\PSL_2(25)$.
Suppose $P \norm S$. Since $P < \gamma_1(S)$ this implies $P=\gamma_2(S)$ and $[S \colon \gamma_2(S)] =5^2$.
Also, every morphism in $\N_{\Aut_\F(P)}(\Aut_S(P))$ is the restriction of a morphism in $\Aut_\F(S)$.  However the normalizer in $\PSL_2(25)$ of a Sylow $5$-subgroup has order $4\cdot 3\cdot 5^2$, contradicting the fact that $|\Aut_\F(S)| = 4\cdot 5^{n-1}$. Thus $P\neq \gamma_2(S)$ and $P$ is not normal in $S$. Also we get $\gamma_2(S)\nleq P$ and by Lemma \ref{Si.in.ess} we deduce that $5^6\leq |S|\leq 5^8$.

\begin{itemize}
\item Suppose $|S|=5^6$.
Then $[\gamma_1(S),\gamma_1(S)] \leq \gamma_{4}(S) =\Z_2(S) \leq \Z(\gamma_1(S)) \leq P$, so $P$ is normal in $\gamma_1(S)$.
If $|P|=5^3$ then $[\gamma_1(S) \colon P] =5^2$, contradicting Theorem \ref{auto.rank.3}. Thus $|P|=5^4$ and $\N_S(P)=\gamma_1(S)$.
Note that $\gamma_2(S)$ has exponent $5$ by Theorem \ref{power} and $[\gamma_2(S),\gamma_2(S)]=[\gamma_2(S),\gamma_3(S)]\leq \gamma_6(S)=1$. So $\gamma_2(S)$ is  elementary abelian. 
The group $\gamma_1(S)$ stabilizes the series $1 < \Z_2(S)=\gamma_{4}(S) < P$. Hence by Lemma \ref{char.series} the group $\Z_2(S)$ is not normalized by $\Aut_\F(P)$. In particular $\Z_2(S) < \Z(P)$ and so $P$ is abelian.
We have $\gamma_1(S)= \gamma_2(S)P$, so $\Z(\gamma_1(S))=\gamma_2(S)\cap P$ has order $5^3$ and since $\Z(\gamma_1(S))\norm S$ we deduce that $\Z(\gamma_1(S))=\gamma_3(S)$.
Thus $\gamma_3(S) \leq \Omega_1(P)$. Since $\gamma_1(S)$ stabilizes the series $1<\gamma_3(S)<P$, the group $\gamma_3(S)$ is not normalized by $\Aut_\F(P)$ and so $P=\Omega_1(P)$. Since $P$ is abelian we deduce that $P$ has exponent $5$ (and so $\gamma_1(S)=\gamma_2(S)P$ has exponent $5$).
Thus $P$ is elementary abelian of order $5^4$.

Set $V=\C_P(O^{5'}(\Aut_\F(P)))$. Then by Lemma \ref{SL.index.p} we get that $P/V$ is a natural $\SL_2(5)$-module for $O^{5'}(\Out_\F(P))\cong \SL_2(5)$. In particular $\Z(S)=[\gamma_1(S),\gamma_1(S)] \nleq V$. 

Let $E\in\mathcal{P}(\F)$ be a pearl and let $\varphi_\lambda\in \Delta_\F(E)$. Lemma \ref{lambda.action} and the fact that $\Z(\gamma_1(S))=\gamma_3(S)$ imply that $\varphi_\lambda$ acts on $\gamma_1(S)/\gamma_3(S)$ as $\begin{pmatrix} \lambda & 0 \\ 0 & 1 \end{pmatrix}$, and so $E$ is abelian. Note that $\varphi_\lambda$ normalizes the quotient $\gamma_1(S)/\gamma_3(S)$ and its maximal subgroup $\gamma_2(S)/\gamma_3(S)$, so we may assume that it normalizes $P/\gamma_3(S)$ and therefore $P$.
 Since $\Z(S)\nleq V$, the morphism $\varphi_\lambda$ acts on $P/V$ as $\begin{pmatrix} \lambda & 0 \\ 0 & \lambda\end{pmatrix}$.  Thus $\varphi_\lambda|_P \notin O^{5'}(\Out_\F(P))$. Recall that $\Out_\F(S) = \langle \varphi_\lambda \rangle$, so there are automorphisms in $\N_{O^{5'}(\Out_\F(P))}(\Out_S(P))$ that are not restrictions of automorphisms of $S$ (but they are restrictions of automorphisms of $\gamma_1(S)=\N_S(P)$). By the Alperin-Goldschmidt fusion theorem this implies that the group $\gamma_1(S)$ is $\F$-essential.
The groups $\Z(S)=[\gamma_1(S),\gamma_1(S)]$ and $\gamma_3(S)=\Z(\gamma_1(S))$ are characteristic in $\gamma_1(S)$ and since $P$ is fully normalized, the group $\gamma_2(S)$ has to be normalized by $\Aut_\F(\gamma_1(S))$. Since $\gamma_1(S)$ is $\F$-essential, the group $O^{5'}(\Out_\F(\gamma_1(S)))$ has a strongly $5$-embedded subgroup and by Lemma \ref{char.series} none of the maximal subgroups of $\gamma_3(S)/\Z(S)\cong \C_5 \times \C_5$ is normalized by $O^{5'}(\Out_\F(\gamma_1(S)))$. Hence the group $O^{5'}(\Out_\F(\gamma_1(S)))$ involves $\SL_2(5)$.

Set $H=  O^{5'}(\Out_\F(\gamma_1(S)))$. Then $H$ is isomorphic to a subgroup of $\SL_4(5)$ (Lemma \ref{GLr}) and satisfies the following:
\begin{itemize}
\item $O_5(H)=1$ ($H$ has a strongly $5$-embedded subgroup);
\item $H$ involves $\SL_2(5)$;
\item $\Out_S(\gamma_1(S))\in \Syl_5(H)$ and $|\N_H(\Out_S(\gamma_1(S)))| \leq 20$ (because every morphism in $\N_{\Out_\F(\gamma_1(S))}(\Out_S(\gamma_1(S)))$ is the restriction of an automorphism of $S$ and $|\Out_\F(S)|=4$).
\end{itemize}
Using the computer software \emph{Magma} we conclude that $H\cong \SL_2(5)$ and $\gamma_3(S)/\Z(S)$ is a natural $\SL_2(5)$-module for $H$.
In particular $|\N_H(\Out_S(\gamma_1(S)))| = 20$ and so $\N_H(\Out_S(\gamma_1(S))) = \N_{\Out_\F(\gamma_1(S))}(\Out_S(\gamma_1(S)))$. This implies $\varphi_\lambda|_{\gamma_1(S)} \in H$. However $\varphi_\lambda$ acts on $\gamma_3(S)/\Z(S)$ as $\begin{pmatrix} \lambda^{3} & 0 \\ 0 & \lambda^2 \end{pmatrix}$, that has determinant not congruent to $1$ modulo $5$ and gives a contradiction.
Therefore $|S|\neq 5^6$.

\item Suppose that $5^7 \leq |S| \leq 5^8$.
Let $l$ be the degree of commutativity of $S$. Then by Theorem \ref{positive.deg} we have $l\geq 1$. In particular by Lemma   \ref{Si.in.ess} we get $\gamma_3(S) \leq P$. Since $P$ is not normal in $S$ we have $\gamma_3(S) < P$ and  so $[\gamma_1(S) \colon P] =5$, that implies $P\norm \gamma_1(S)$. Note that $\Omega_1(\gamma_1(S))=\Z_4(S) \leq \gamma_3(S)\leq P$, so $\Omega_1(\gamma_1(S))=\Omega_1(P)$. The group $\gamma_2(S)$ is not contained in $P$, normalizes $P$ and stabilizes the series $\Z(S) < \Omega_1(P)=\Z_4(S) < P$.
We show that $\Z(S) \leq \Phi(P) \leq \Z_4(S)$, contradicting Lemma \ref{char.series}. Note that either $\Z_4(S)=\gamma_3(S)$ or $\Z_4(S)=\gamma_{4}(S)$. Since $[\gamma_1(S),\gamma_1(S)]=[\gamma_1(S),\gamma_2(S)]\leq \gamma_{4}(S)$ we conclude that $P$ centralizes $P/\Z_4(S)$. Also, by Lemma \ref{power} we have $\gamma_1(S)^5 = \gamma_5(S) \leq \Z_4(S)$. So $P^5\leq \Z_4(S)$ and $\Phi(P)\leq \Z_4(S)$.
Suppose $\Z_3(S)\leq \Z(P)$. Since $\Z_3(S)\leq \Z(\gamma_2(S))$ and $\gamma_1(S)=\gamma_2(S)P$, we get $\Z_3(S) \leq \Z(\gamma_1(S))$. Thus $\Z_3(S) \leq \Omega_1(\Z(P)) \leq \Omega_1(P) =\Z_4(S)$. Hence $\Omega_1(\Z(P))$ is normal in $S$. Since $P$ is not normal in $S$ we have $\gamma_1(S)=PP^g$ for some $g\in S$ and $\Omega_1(\Z(\gamma_1(S))) =\Omega_1(\Z(P)) \cap \Omega_1(\Z(P^g)) =\Omega_1(\Z(P))$. In particular $\gamma_1(S)$ stabilizes the series $\Omega_1(\Z(P)) \leq \Z_4(P)=\Omega_1(P) <P$, contradicting Lemma \ref{char.series}. Therefore $\Z_3(S)\nleq \Z(P)$ and so $[\Z_3(S),E] =\Z(S)$. Thus $\Z(S) \leq [E,E] \leq \Phi(E) \leq \Z_4(S)$ and we reach a  contradiction.
\end{itemize}
\end{proof}

\begin{proof}[\textbf{Proof of Theorem \ref{small.rank}}]~
Suppose $|S|=p^{k+1}$. Then by Theorem \ref{main} the group $S$ has an elementary abelian maximal subgroup. Also, if $k=2$ then $S\cong p^{1+2}_+$ and if $k=3$ then by Lemma \ref{Sp4p} the group $S$ is isomorphic to a Sylow $p$-subgroup of the group $\Sp_4(p)$.
Finally  if $k\geq 3$ then by Lemma \ref{pearl} the group $\gamma_1(S)$ is the only candidate for an $\F$-essential subgroup that is not a pearl.

Suppose $|S|\neq p^{k+1}$. Then by Lemmas \ref{small.rank.1}, \ref{sec.3.exotic} and \ref{7.6}  one of the following holds:
\begin{itemize}
\item $(p,k) \in \{ (2,3), (4,5) \}$, $\gamma_1(S)=\C_S(\Z_2(S))$ and the $\F$-essential subgroups of $S$ are given by \cite[Theorem 1.1]{DRV} if $p=3$ and by Lemma \ref{4.5} if $p=5$.

\item  $k =3$, $p=7$, $S$ has order $7^5$, $S\cong$ \rm{\texttt{SmallGroup(7\string^5, 37)}}, $\mathcal{P}(\F)= \mathcal{P}(\F)_a$ and there exists an abelian pearl $E\in \mathcal{P}(\F)_a$ such that
$E^\F$ is the unique $\F$-conjugacy class of $\F$-essential subgroup of $S$, $\Aut_\F(E) \cong \SL_2(7)$, $\Out_\F(S)\cong \C_6$, $\F$ is simple and if we assume the classification of finite simple groups then $\F$ is exotic.

\item $k=4$, $p=7$, $|S|=7^6$ and

\begin{itemize}
\item if $\mathcal{P}(\F)_a \neq \emptyset$ then $S$ is isomorphic to a Sylow $7$-subgroup of the group $\G_2(7)$ and the $\F$-essential subgroups of $S$ are described in \cite[Theorem 4.2]{G2p};
\item if $\mathcal{P}(\F)=\mathcal{P}(\F)_e$ then $S\cong$  \rm{\texttt{SmallGroup(7\string^6, 813)}} and there exists a pearl $E\in \mathcal{P}(\F)_e$ such that $E^\F$ is the unique $\F$-conjugacy class of $\F$-essential subgroups of $S$. In particular $O_7(\F) = \Z(S)$. Note that in this case $S/\Z(S)\cong$ \rm{\texttt{SmallGroup(7\string^5, 37)}}, $E/\Z(S)$ is an abelian pearl for $\F/\Z(S)$ and $\F/\Z(S)$ is the unique saturated fusion system defined on $S/\Z(S)$ and containing an abelian pearl.
\end{itemize}
\end{itemize}
\end{proof}

\begin{remark}
As we mentioned in the introduction of this paper, the classification of saturated fusion systems containing pearls on $p$-groups of sectional rank $p-1$ will be the subject of a forthcoming paper.
For example, suppose that $\F_0$ is a saturated fusion system on the $5$-group $X_0$ that has order $5^5$ and sectional rank $4$ and suppose that $\F_0$ contains a pearl $E$. By Theorem \ref{main} the group $X_0$ has a maximal subgroup that is elementary abelian. Suppose there exists a tower of saturated fusion systems $\F_0 \subset \F_1 \subset \dots \subset \F_n \subset \dots$ defined on the $5$-groups
\[X_0 < X_1 < \dots < X_n < \dots\] such that for every $i\geq 0$ the group $E$ is a pearl of $\F_i$ and the $5$-group $X_i$ has sectional rank $4$ and is a maximal subgroup of $X_{i+1}$. Suppose moreover that $X_1$ does not have index $5$ abelian subgroups. Then there are $5$ candidates for $X_1$, namely the groups stored in \emph{Magma} as \texttt{SmallGroup(5\string^6,i)} for $i\in \{636, 639, 640, 641, 642\}$. Among these ones, only the group \texttt{SmallGroup(5\string^6,636)} is contained in a $5$-group of maximal nilpotency class and order $5^7$ containing a pearl. In other words, there are $5$ towers of $5$-groups containing $X_0$ such that $X_1$ does not have index $5$ abelian subgroups and $4$ of these have only two members: $X_0 < X_1$. Inspired by this observation, we claim that typically a $5$-group of sectional rank $4$ containing a pearl has an index $5$ abelian subgroup and only a finite number of examples \emph{deviates} from such standard case.

\end{remark}

\section*{Acknowledgments}
The main theorems of this paper are generalizations of results proved by the author in her PhD thesis, under the supervision of Prof. Chris Parker. She is immensely grateful to him for his support. She would also like to show her gratitude to Dr. Ellen Henke for comments that improved this manuscript.

\bibliographystyle{alpha}
\bibliography{bibpearls}

\end{document}